 \documentclass[final,3p,times]{elsarticle}

\usepackage{amssymb}
 \usepackage{amsthm}
\usepackage{amsmath,amssymb,amsopn,amsfonts,mathrsfs,amsbsy,amscd}

\newcommand{\prs}{\langle\;,\;\rangle}

\newcommand{\too}{\longrightarrow}
\newcommand{\om}{\omega}
\newcommand{\esp}{\quad\mbox{and}\quad}

\def\br{[\;,\;]}

\newcommand{\X}{{\cal X}}
\newcommand{\G}{{\mathfrak{g}}}

\newcommand{\tr}{{\mathrm{tr}}}

\newcommand{\D}{{\cal D}}

\newcommand{\di}{\displaystyle}
\newcommand{\Om}{\Omega}
\newcommand{\na}{\nabla}

\newcommand{\al}{\alpha}
\newcommand{\be}{\beta}
\newcommand{\ga}{\gamma}
\newcommand{\Ga}{\Gamma}
\newcommand{\e}{\epsilon}

\newcommand{\De}{\Delta}

\font\bb=msbm10

\def\R{\hbox{\bb R}}

\def\S{\hbox{\bb S}}

\def\C{\hbox{\bb C}}
\newtheorem{Def}{Definition}[section]
\newtheorem{theo}{Theorem}[section]
\newtheorem{pr}{Proposition}[section]
\newtheorem{Le}{Lemma}[section]
\newtheorem{co}{Corollary}[section]

\newtheorem{exem}{Example}
\newtheorem{remark}{Remark}

\newtheorem{problem}{Problem}

\begin{document}

\begin{frontmatter}


 
 \fntext[label3]{This research was conducted within the framework of Action concert\'ee CNRST-CNRS Project SPM04/13.}

\title{On para-K\"ahler  Lie algebroids and generalized pseudo-Hessian structures}

 \author[label1,label2]{Sa\"id Benayadi, Mohamed Boucetta}
 \address[label1]{Universit\'e de Lorraine, Laboratoire IECL, CNRS-UMR 7502,\\ Ile du Saulcy, F-57045 Metz
 cedex
 1, France.\\e-mail: said.benayadi@univ-lorraine.fr
 }
 \address[label2]{Universit\'e Cadi-Ayyad\\
 Facult\'e des sciences et techniques\\
 BP 549 Marrakech Maroc\\e-mail: m.boucetta@uca.ac.ma
 }



\begin{abstract}In this paper, we generalize all the results obtained on para-K\"ahler Lie algebras in \cite{bouben} to para-K\"ahler Lie algebroids. In particular, we study exact para-K\"ahler Lie algebroids as a generalization of exact para-K\"ahler Lie algebras. This study leads to a natural generalization of pseudo-Hessian manifolds. Generalized pseudo-Hessian manifolds have many similarities with Poisson manifolds. We explore these similarities which, among others, leads to  a powerful machinery to build examples of non trivial   pseudo-Hessian structures. Namely, we will show that given a finite dimensional commutative and associative algebra $(\mathcal{A},.)$, the orbits of the action $\Phi$ of $(\mathcal{A},+)$ on $\mathcal{A}^*$ given by $\Phi(a,\mu)=\exp(L_a^*)(\mu)$ are pseudo-Hessian manifolds, where $L_a(b)=a.b$.  We illustrate this result by considering many examples of associative commutative algebras an show that the pseudo-Hessian manifolds obtained are very interesting.
\end{abstract}

\begin{keyword}para-K\"ahler Lie algebroids \sep  symplectic Lie
algebroids \sep left symmetric algebroids \sep pseudo-Hessian manifolds \sep Associative commutative algebras
\MSC 53C15 \sep \MSC 53A15 \sep \MSC 53D17 \sep \MSC 13P25


\end{keyword}

\end{frontmatter}






\section{Introduction}\label{section1}

Recall that a Lie algebroid is a vector bundle $A\too M$ together with an anchor map $\rho:A\too TM$ and a Lie bracket $[\;,\;]_A$ on $\Ga(A)$ such that, for any $a,b\in \Ga(A)$, $f\in C^\infty(M)$,
\[ [a,fb]_A=f[a,b]_A+\rho(a)(f)b. \]
Lie algebroids are now a central notion in
differential geometry and constitute an active domain of research.
They have many applications in various part of mathematics and physics (see
for instance \cite{can,car,cor,mac}). It is a well-established fact that many classical geometrical structures involving the tangent bundle of a manifold (which has a natural structure of Lie algebroid) can be generalized to the context of Lie algebroids. Thus the notions of connections on Lie algebroids, symplectic Lie algebroids, pseudo-Riemannian Lie algebroids and so on are now usual notions in differential geometry with many applications in physics (see for instance \cite{blu, elyasi}). On the other hand, it is important to point out that Lie algebroids generalize also Lie algebras and, for instance, if one obtains a result on  the curvature of pseudo-Riemannian Lie algebroids this result holds for the curvature of pseudo-Euclidean Lie algebras and hence  for the curvature of left invariant pseudo-Riemannian metrics on Lie groups. 

In this paper, we study para-K\"ahler Lie algebroids as a generalization of both para-K\"ahler manifolds and left invariant para-K\"ahler structures on Lie groups.  A
\emph{para-K\"ahler} structure on a manifold $M$ is a pair $(g,K)$ where $g$ is a
pseudo-Riemannian metric and $K$ is a parallel (with respect to the Levi-Civita connection of $g$) skew-symmetric endomorphism field satisfying $K^2=Id_{TM}$. The
paper \cite{cru} contains a survey on  para-K\"ahler geometry and contains many references. When the manifold is a Lie group
$G$,  the metric and the para-complex structure are considered left-invariant, they
are both determined by their restrictions to the Lie algebra $\G$ of $G$.  In a such
situation, $(\G,g_e,K_e)$ is called   \emph{para-K\"ahler  Lie
	algebra}. A para-K\"ahler Lie algebroid is a Lie algebroid $(A,M,\rho)$ together with a pseudo-Euclidean product $\prs$  on $A$ and a bundle isomorphism $K:A\too A$ such that $K^2=\mathrm{Id}_A$, $K$ is skew-symmetric with respect to $\prs$ and $\na K=0$ where $\na$ is the Levi-Civita connection associated to $\prs$ and given by the formula
\begin{eqnarray*}
	2\langle\na_ab,c\rangle&=&{\rho}(a).\langle b,c\rangle+{\rho}(b).\langle a,c\rangle-
	{\rho}(c).\langle a,b\rangle
	+\langle[c,a]_A,b\rangle+\langle[c,b]_A,a\rangle+\langle[a,b]_A,c\rangle,\quad a,b,c\in\Ga(A).\end{eqnarray*}
The authors realized a complete study of para-K\"ahler  Lie
algebras in \cite{bouben} and, our first motivation at the origin of the present paper,  was to generalize the result obtained in this study  to the context of para-K\"ahler Lie algebroids. This has been done successfully and constitutes the first part of this paper. The generalization was not straightforward and many new phenomenas appeared due to the anchor. Moreover, as it happens always in mathematics, during our investigations when studying a special class of para-K\"ahler Lie algebroids, we came  across a new structure which turned out to be a natural generalization of the notion of pseudo-Hessian manifolds. This new notion and some of its remarkable properties constitute the second part of this paper. Recall that a pseudo-Hessian manifold is a locally affine manifold $(M,D)$ endowed with a pseudo-Riemannian metric such that $g$ is locally given by $D d\phi$  where $\phi$ is a function. This is equivalent to $S=Dg$ is totally symmetric. Pseudo-Hessian geometry is an active domain of research which has many applications in economic theory, in system modeling and optimization as well as in statistical theory. One can consult  \cite{cortes, shima} to find out more about this geometry and its origins. To summarize, in this paper, we generalize all the results obtained in \cite{bouben} to para-K\"ahler Lie algebroids, and  we introduce a natural generalization of pseudo-Hessian manifolds, we them call generalized pseudo-Hessian manifolds. Let us give briefly the definition of this structure and some of its striking properties. A generalized pseudo-Hessian manifold is triple $(M,D,h)$ where $(M,D)$ is a locally affine manifold and $h$ is a symmetric bivector field such that the tensor $T\in\Ga(\otimes^3 TM)$ given by $T(\al,\be,\ga)=D_{h_\#(\al)}h(\be,\ga)$ is totally symmetric, where $h_\#:T^*M\too TM$ is given by $\be(h_\#(\al))=h(\al,\be)$. When $h$ is invertible and of constant signature (for instance when $M$ is connected), $(M,D,h^{-1})$ is a pseudo-Hessian manifold.  There are many similarities between Poisson manifolds as a generalization of symplectic manifolds and generalized pseudo-Hessian manifolds as a generalization of pseudo-Hessian manifolds. Indeed, if $(M,D,h)$ is a pseudo-Hessian manifold then $\mathrm{Im}h_\#$ is an integrable distribution and defines a singular foliation whose leaves are pseudo-Hessian manifolds. There is an analogue of Darboux-Weinstein theorem (see Theorem \ref{local}) and the bracket $[\;,\;]_h$ on $\Om^1(M)$ and $\D:\Om^1(M)\times\Om^1(M)\too \Om^1(M)$ given by
\begin{equation*} \prec\D_\al\be,X\succ=\na_X h(\al,\be)+\prec\na^*_{h_\#(\al)}\be,X\succ\esp [\al,\be]_\D=
\na^*_{h_\#(\al)}\be-\na^*_{h_\#(\be)}\al, \end{equation*}satisfy $(T^*M,M,h_\#,[\;,\;]_h)$ is a Lie algebroid and $\D$ is a torsionless flat connection for this Lie algebroid. Moreover, for any $x\in M$, $\mathcal{A}_x=\ker h_\#(x)$ carries a natural structure of commutative associative algebra. On the other hand, let $(\mathcal{A},.)$ be a commutative associative algebra $(\mathcal{A},.)$. Denote by $D$ the canonical affine connection on $\mathcal{A}^*$ and define the symmetric bivector field $h$ on  $\mathcal{A}^*$ by
\begin{equation*} h(\al,\be)(\mu)=\prec\mu,\al(\mu).\be(\mu)\succ,\quad \al,\be\in\Om^1(\mathcal{A}^*)=C^\infty(\mathcal{A}^*,\mathcal{A}),\mu\in \mathcal{A}^*. \end{equation*}
Then $(\mathcal{A}^*,D,h)$ is a generalized pseudo-Hessian manifolds and the leaves of the foliation associated to $\mathrm{Im}h_\#$ are the orbits of the action $\Phi$ of $(\mathcal{A},+)$ on $\mathcal{A}^*$ given by $\Phi(a,\mu)=\exp(L_a^*)(\mu)$  where $L_a(b)=a.b$ (See Theorem \ref{main}). Thus the orbits of $\Phi$ are pseudo-Hessian manifolds. This give powerful machinery to build examples of pseudo-Hessian structures. We will show that the pseudo-Hessian structure of these orbits is not trivial since their Hessian curvature is not zero. We illustrate this result by considering many examples of associative commutative algebras an show that the pseudo-Hessian structures obtained on the orbits of $\Phi$ are very interesting.

We give now the organization of this paper. In Section \ref{section2}, we recall some basic fact about Lie algebroids and connections on Lie algebroids. A Lie algebroid with a torsionless and flat connection was called left symmetric algebroid in \cite{bai2}. These algebroids play a central role in the study of para-k\"ahler Lie algebroids. We adopt the terminology of left symmetric algebroids as in \cite{bai2} and we give some of their geometrical properties we will use later. In Section \ref{section3}, we start the study of para-K\"ahler Lie algebroids and the main result here is Theorem \ref{theoextendible} which states that a para-K\"ahler Lie algebroids is obtained from two left symmetric algebroids on two dual vector bundles compatible in some sense. In Section
\ref{section4} and \ref{section5}, we study exact para-K\"ahler Lie algebroids and the related notions of $\S$-matrices and quasi-\S-matrices on a left symmetric algebroid. If $(M,D)$ is an affine manifold then $(TM,M,D)$ becomes a left symmetric algebroid and a symmetric $\S$-matrix on $(TM,M,D)$ defines a generalized pseudo-Hessian structure on $(M,D)$. Section  \ref{section6} is devoted to the study of this new structure. In Section \ref{section7}, we study linear pseudo-Hessian manifolds and we give many examples.

{\bf Notations:} Let  $A\too M$ be a vector bundle and $F:A\too A$  a bundle endomorphism.  We denote by $\Ga(A)$ the space of its sections and by $F^*:A^*\too A^*$  the dual endomorphism. For any $X\in A_x$ and $\al\in A^*_x$, we denote $\al(X)$ by $\prec\al,X\succ$.
 The phase space of $A$ is the vector bundle $\Phi(A):=A\oplus A^*$ endowed with the two nondegenerate bilinear forms $\prs_0$ and $\Om_0$ given by
 \begin{equation*}\label{fdual}
 \langle u+\al,v+\be\rangle_0=\prec\al,v\succ+\prec\be,u\succ\esp
 \Om_0(u+\al,v+\be)=\prec\be,u\succ-\prec\al,v\succ.
 \end{equation*}We denote by $K_0:\Phi(A)\too\Phi(A)$ the bundle endomorphism given by $ K_0(u+\al)=u-\al$.  \\
 Let $\om\in\Ga(\wedge^2A^*)$ which is nondegenerate. We denote by $\flat:A\too A^*$ the bundle isomorphism given by $\flat(v)=\om(v,.)$.

\section{Lie algebroids, connections, Levi-Civita connections, left symmetric algebroids and symplectic Lie algebroids}\label{section2}

Through this paper, we will use some well-known  basic notions,  namely, anchored bundles, Lie algebroids, connections on Lie algebroids and symplectic Lie algebroids. In this section, we recall the definitions of these notions, we give some of their properties and some basic examples. For more details one can consult \cite{car, fer1, mac}. We introduce also  left symmetric algebroids generalizing left symmetric algebras. They are Lie algebroids which will play a central role in the study of para-K\"ahler Lie algebroids.

\paragraph{Lie algebroids and their immediate properties} An {\it anchored vector bundle} is a triple $(A,M,\rho)$ where  $p:A\too M$ is a vector bundle and $\rho:A\too TM$ is a bundle homomorphism called anchor. An homomorphism between two anchored vector bundles $(A,M,\rho)$ and $(A',M',\rho')$ is a bundle homomorphism $\phi:A\too A'$ such that $\rho=\rho'\circ \phi$. The sum of two anchored bundles $(A,M,\rho)$ and $(B,M,\rho')$ is the anchored bundle $(A\oplus B,M,\rho\oplus\rho')$.\\
Let $(A,M,\rho)$ be an anchored vector bundle. A bracket  on $\Ga(A)$ is a skew-symmetric $\R$-bilinear map $\br_A:\Ga(A)\times\Ga(A)\too\Ga(A)$. It
is called {\it anchored} if
 for any $a,b\in\Ga(A)$ and for every smooth
function $f\in C^\infty(M)$,
\begin{equation}\label{eq1}[a,fb]_A=f[a,b]_A+\rho(a)(f)b.\end{equation}
By using a classical argument, we can deduce from
this relation  that $\br_A$ is local in the sense that if a section $a$ vanishes on an open set $U$ then for any $b\in\Ga(A)$, $[a,b]_A$ vanishes on $U$. The {\it torsion} of $\br_A$ is the map $\tau_{\br_A}:\Ga(A)\times\Ga(A)\too\mathcal{X}(M)$ given by
\begin{equation}\label{tau}
\tau_{\br_A}(a,b)=\rho([a,b]_A)-[\rho(a),\rho(b)].
\end{equation}$\tau_{\br_A}$ is $\R$-bilinear, skew-symmetric and, for any $f\in C^\infty(M)$,
\[ \tau_{\br_A}(fa,b)=\tau_{\br_A}(a,fb)=f\tau_{\br_A}(a,b). \] So $\tau_{\br_A}\in\Ga( \wedge^2A^*\otimes TM)$. In order to study under which conditions $\br_A$ is a Lie bracket, we introduce
the Jacobiator of $\br_A$ as $J_{\br_A}:\Ga(A)\times\Ga(A)\times\Ga(A)\too
 \Ga(A)$ given by
 \[ J_{\br_A}(a,b,c)=[[a,b]_A,c]_A+[[b,c]_A,a]_A+[[c,a]_A,b]_A. \]
 $J$ is $\R$-trilinear and  skew-symmetric. Thus $\br_A$ is a Lie bracket if and only if $J_{\br_A}=0$. However, this equation is not tensorial and may be very difficult to check in concrete situations. Nevertheless, 
   for any $a,b,c\in\Ga(A)$ and any  $f\in C^\infty(M)$,
 \begin{equation}\label{jacobiator}
 J_{\br_A}(a,b,fc)=fJ_{\br_A}(a,b,c)+\tau_{\br_A}(a,b)(f)c.
 \end{equation}
 This relation shows that $J_{\br_A}$ is local and if $\tau_{\br_A}$ vanishes then $J_{\br_A}$ becomes a tensor, namely, $J_{\br_A}\in\Ga(\wedge^3A^*\otimes TM)$. This shows also that if $J_{\br_A}$ vanishes then $\tau_{\br_A}$ does.  The following proposition is an immediate consequence of \eqref{jacobiator} and give us an useful way of checking if an anchored bracket is actually a Lie bracket.

 \begin{pr}\label{pr1} Let $(A,M,\rho)$ be an anchored bundle and $\br_A$  an anchored bracket on $\Ga(A)$. Then the following assertions are equivalent:
 \begin{enumerate}\item[$(i)$] $(\Ga(A),\br_A)$ is a Lie algebra, i.e., $J_A$ vanishes identically.
 \item[$(ii)$] For any $x\in M$ there exists an open set $U$ of $M$ containing $x$ and a basis of sections $(a_1,\ldots,a_r)$ over $U$ such that
  \[ J_{\br_A}(a_i,a_j,a_k)=0\esp \tau(a_i,a_j)=0,\quad 1\leq i<j<k\leq r. \]
  
 \end{enumerate}

 \end{pr}
 \begin{Def}
  A Lie algebroid is an anchored vector bundle $(A,M,\rho)$ together with an  anchored bracket $\br_A$ satisfying $(i)$ or $(ii)$ of Proposition \ref{pr1}. A Lie algebroid is called transitive if its anchor map is onto on every point and it is called regular if the rank of $\rho$ is constant on $M$.
  \end{Def}

 There are some well-known properties  of a Lie algebroid $(A,M,\rho,[\;,\;]_A)$.
\begin{enumerate}\item[$(a)$]  The induced map $\rho:\Ga(A)\too\X(M)$ is a Lie algebras
homomorphism.
\item[$(b)$] The smooth distribution $\mathrm{Im}\rho$ is integrable in the sense of Sussmann \cite{sus} and, for any leaf $L$ of $\mathrm{Im}\rho$, $(A_{|L},L,\rho,[\;,\;]_A)$ is a transitive Lie algebroid.
\item[$(c)$] For any $x\in
M$, there is an induced Lie bracket say $\br_x$ on
${\G}_x=\ker(\rho_x)\subset A_x$ which makes it
into a finite dimensional Lie algebra.
\item[$(d)$] The map $d_A:\Ga(\wedge A^*)\too\Ga(\wedge A^*)$ by
\[ d_AQ(a_1,\ldots,a_p)=\sum_{i=1}^p(-1)^{i+1} \rho(a_i).Q(a_1,\ldots,\hat{a}_i,\ldots,a_p)
-\sum_{1\leq i<j\leq p}(-1)^{i+j+1}Q([a_i,a_j]_A,a_1,\ldots,\hat{a}_i,\ldots,\hat{a}_j,\ldots,a_p), \] is a differential, i.e., $d_A^2=0$.
In particular, for any $a,b\in\Ga(A)$, $f\in C^\infty(M)$ and  $Q\in\Ga(\wedge A^*)$,
\[ d_Af(a)=\rho(a)(f)\esp
 d_AQ(a,b)=\rho(a).Q(b)-\rho(b).Q(a)-Q([a,b]_A). \]
 These two relations show that there is a correspondence between  Lie algebroids structure on $(A,M)$ and differentials on  $\Ga(\wedge A^*)$.
\item[$(e)$] The bracket $\br_A$ extends to a new bracket denoted in the same way on the sections of $\wedge A=A\oplus\ldots\oplus\wedge^{\mathrm{rank}A}A$ called {\it generalized Schouten-Nijenhuis bracket}. Its
properties are the same as those of the usual Schouten-Nijenhuis bracket and the anchor extends also to give a map $\rho:\wedge A\too\wedge TM$ which preserves Schouten-Nijenhuis brackets. It is important here to point out that if $\Pi\in\Ga(\wedge^2A)$ satisfies $[\Pi,\Pi]_A=0$ then $\pi=\rho(\Pi)$ is a Poisson tensor on $M$.
\end{enumerate}

\paragraph{Some examples of Lie algebroids} \begin{enumerate}\item The basic example of a Lie algebroid over $M$ is the tangent bundle itself, with the identity mapping as anchor.
	\item Every finite dimensional Lie algebra is a Lie algebroid over a one point space. 
	\item Any integrable subbundle of $TM$ is a Lie algebroid with the inclusion as anchor and the induced bracket.\item Let $(M,\pi)$ be a Poisson manifold. The bivector field $\pi$ defines a bundle homomorphism $\pi_{\#}:T^*M\too TM$ and a bracket on $\Om^1(M)$ by
	\[ [\al,\be]_\pi=\mathcal{L}_{\pi_{\#}(\al)}\be- \mathcal{L}_{\pi_{\#}(\be)}\al-d\pi(\al,\be)\]such that $(T^*M,M,\pi_{\#},\br_\pi)$ is a Lie algebroid.\item
	Let $\G\stackrel{\tau}\too\X(M)$ be an action of a finite-dimensional real
	Lie algebra $\G$ on a smooth manifold $M$,
	i.e., a morphism of Lie algebras
	from $\G$ to the Lie algebra of vector fields
	on $M$. Consider $(A,M,\rho,\br_A)$, where $A=M\times\G$ as a trivial bundle and
	\[ \rho((m,\xi))=\tau(\xi)(m)\esp    [\xi,\eta]_A=\mathcal{L}_{\rho(\xi)}\eta-
	\mathcal{L}_{\rho(\eta)}\xi+[\xi,\eta]_\G,\quad\eta,\xi\in\Ga(A)=C^\infty(M,\G). \]By using $(ii)$ of Proposition \ref{pr1}, it is easy to check that $(A,M,\rho,\br_A)$ is a Lie algebroid.
\end{enumerate}

\paragraph{Connections on Lie algebroids} Given a Lie algebroid $(A,M,\rho,\br_A)$,  an $A$-connection on a vector bundle $E\too M$  is
a $\R$-bilinear operator $\na:\Ga(A)\times\Ga(E)\too\Ga(E)$  satisfying:
 $$\na_{fa}s=f\na_a s\esp 
\na_a(fs)=f\na_a s+\rho(a)(f)s,$$
 for any
$a\in\Ga(A)$, $s\in\Ga(E)$ and $f\in C^\infty(M).$ We shall call  $A$-connections on the vector bundle $A\too M$
 \emph{ linear $A$-connections}. The
 curvature of an $A$-connection $\na$ on $E$ is formally identical
 to the usual definition
 $$R(a,b)s=\na_{a}\na_{b}s-\na_{b}\na_{a}s-\na_{[a,b]_A}s,$$
 where $a,b\in\Ga(A)$ and $s\in\Ga(E)$. The connection $\na$ is
 called flat if $R$ vanishes identically.  The dual of  $\na$  is the $A$-connection $\na^*$ on $E^*$ given by
  \begin{equation}\label{dual} \prec\na_a^*\al,s\succ=\rho(a).\prec\al,s\succ-\prec\al,\na_as\succ, \end{equation}
 for any $a\in\Ga(A),s\in\Ga(E),\al\in\Ga(E^*)$.
 
 There is a notion of parallel transport associated to an $A$-connection $\na$ on $q:E\too M$. An $A$-path is a curve $\al:[a,b]\too A$ such that
 \[ \forall t\in[a,b],\quad \rho(\al(t))=c'(t), \]where $c=p\circ\al:[a,b]\too M$. Given an $A$-path $\al:[a,b]\too A$, we denote by $\Ga_\al(E)$ the vector space of curves $s:[a,b]\too E$ such that $q\circ s=p\circ\al$. The connection $\na$ defines a unique derivative $\na^\al:\Ga_\al(E)\too \Ga_\al(E)$ and a parallel transport  $\tau_{\al}:E_{c(a)}\too E_{c(b)}$ given by $\tau_\al(s_1)=s(a)$ where $s\in\Ga_\al(E)$ is uniquely determined by $\na^\al s=0$ and $s(a)=s_1$.

 \paragraph{The canonical connection on the adjoint bundle}  Let $(A,M,\rho,\br_A)$ be a Lie algebroid such that $\rho$ has a constant rank over $M$. Then {\it the adjoint bundle} ${\G}=\ker\rho\too M$ is a vector bundle of Lie algebras. Define $\na^\G:\Ga(A)\times\Ga({\G })\too\Ga({\G})$ by
\begin{equation}\label{connection}
\na_a^\G s=[a,s]_A.
\end{equation}This defines a  $A$-connection on ${\G}$ satisfying, for any $a,b\in\Ga(A)$, $s_1,s_2\in\Ga(\G)$,
\begin{equation}\label{connection1} R^{{\na}^\G}(a,b)=0\esp
\na_a^\G[s_1,s_2]_A=[\na_a^\G s_1,s_2]_A+[s_1,\na_a^\G s_2]_A,
\end{equation}for any $a\in \Ga(A)$ and $s_1,s_2\in \Ga({\G})$. For any $x,y\in M$ lying in the same leaf of the characteristic foliation, there exists an $A$-path $\al:[a,b]\too A$ such that $p\circ\al(a)=x$ 	and $p\circ\al(b)=y$. The second relation in \eqref{connection1} shows that the parallel transport $\tau_\al:\G_x\too\G_y$ is an isomorphism of Lie algebras.

\paragraph{The Levi-Civita connection of a pseudo-Riemannian Lie algebroid}  A pseudo-Riemannian metric of signature $(p,q)$ on a Lie algebroid $(A,M,\rho,\br_A)$ is the data,
for any $x\in M$, of  a nondegenerate product $\prs_x$ of signature $(p,q)$ on the fiber $A_x$
such that, for any local sections $a,b$ of $A$, the function
$\langle a,b\rangle$ is smooth. A Lie algebroid with a pseudo-Riemannian metric is called {\it pseudo-Riemannian Lie algebroid.}\\
The most interesting fact about pseudo-Riemannian  Lie
algebroids is the existence on the analogous of the Levi-Civita
connection. Indeed, if $\prs$ is a pseudo-Riemannian metric on a Lie
algebroid $(A,M,\rho,\br_A)$ then the formula
\begin{eqnarray*}
	2\langle\na_ab,c\rangle&=&{\rho}(a).\langle b,c\rangle+{\rho}(b).\langle a,c\rangle-
	{\rho}(c).\langle a,b\rangle
	+\langle[c,a]_A,b\rangle+\langle[c,b]_A,a\rangle+\langle[a,b]_A,c\rangle\end{eqnarray*}defines
a linear $A$-connection  which is characterized by the two following properties:
\begin{enumerate}
	\item[$(i)$] $\na$ is metric, i.e., $\rho(a).\langle b,c\rangle=\langle\na_ab,c\rangle+\langle b,\na_ac\rangle$,
	
	\item[$(ii)$] $\na$ is torsion free, i.e., $\na_ab-\na_ba=[a,b]_A.$\end{enumerate}
We  call $\na$ the \emph{ Levi-Civita $A$-connection} associated to
$\prs$. Moreover, $\prs$ and $\rho$ define
a symmetric bivector field $h\in\Ga(TM\otimes TM)$  by
\begin{equation}\label{metric} h(\al,\be)=\langle \#\circ\rho^*(\al),\#\circ\rho^*(\be)\rangle= \prec\be,\rho\circ\#\circ\rho^*(\al)\succ=
\prec\al,\rho\circ\#\circ\rho^*(\be)\succ, \end{equation}
where $\#:A^*\too A$ is the isomorphism associated to $\prs$. The following relations are easy to check:
\[ \G_x^\perp=\#\circ\rho^*(T_x^*M),\; \G_x\cap \G_x^\perp=\#\circ\rho^*(\ker h_x)\esp \ker\rho^*_x\subset\ker h_x,\; x\in M,\;\G_x=\ker\rho_x. \]
These relations show that $h_x$ is nondegenerate if and only if $\G_x$ is $\prs$-nondegenerate and $\rho_x$ is onto.

\paragraph{Symplectic Lie algebroids} A \emph{symplectic } form on a Lie algebroid   $(A,M,\rho,\br_A)$ is
a bilinear skew-symmetric nondegenerate form $\om\in\Ga(\wedge^2A^*)$ such that, for any $a,b,c\in\Ga(A)$,
\begin{equation*}\label{eq14}d_A\om(a,b,c):=
\rho(a).\om(b,c)+\rho(b).\om(c,a)+\rho(c).\om(a,b)- \om([a,b]_A,c)-\om([b,c]_A,a)-\om([c,a]_A,b)=0.
\end{equation*} We call $(A,M,\rho,\br_A,\om)$ a {\it symplectic Lie algebroid}. There is a natural Poisson structure on $M$ associated to $(A,M,\rho,\br_A,\om)$.
\begin{pr}\label{prpoisson} Let $(A,M,\rho,\br_A,\om)$ be a symplectic Lie algebroid and $\flat:A\too A^*$, $a\mapsto \om(a,.)$. Then the bivector field $\pi$ given, for any
	$\al,\be\in\Om^1(M)$, by
	\begin{equation}\label{poisson}
	\pi(\al,\be)=\om(\flat^{-1}\circ\rho^*(\al),
	\flat^{-1}\circ\rho^*(\be))=\prec\be,\rho\circ\flat^{-1}\circ\rho^*(\al)\succ,
	\end{equation}is a Poisson tensor. Moreover, $\flat^{-1}\circ\rho^*:T^*M\too A$  is a Lie algebroid homomorphism, where  $T^*M$ is endowed with the Lie algebroid structure associated to $\pi$.
	
\end{pr}

\begin{proof}
	The inverse of $\om$ is an element $\Pi\in\Ga(\wedge^2A)$ and the condition $d_A\om=0$ is equivalent to $[\Pi,\Pi]_A=0$,  the bracket here is the generalized Schouten-Nijenhuis bracket.  The image $\pi=\rho(\Pi)$ where  $\rho:\wedge^2A\too\wedge^2TM$  is a Poisson tensor on $M$. Moreover, $\Pi$ defines on $\Ga(A^*)$ a Lie bracket such that $\flat^{-1}:A^*\too A$ is a morphism of Lie algebroids. On the other hand, since $\pi=\rho(\Pi)$, $\rho^*:T^*M\too A^*$ is a morphism of Lie algebroids. Finally,
	it is easy to check that $\pi$ is given by \eqref{poisson}.\end{proof}

Let $(A,M,\rho,\br_A,\om)$ be a symplectic Lie algebroid, as in the case of a pseudo-Riemannian Lie algebroid, we have
\begin{equation}\label{important} \G_x^\om=\flat^{-1}\circ\rho^*(T_x^*M),\; \G_x\cap \G_x^\om=\flat^{-1}\circ\rho^*(\ker \pi_x)\esp \ker\rho^*_x\subset\ker \pi_x,\; x\in M,\;\G_x=\ker\rho_x, \end{equation}where $\G_x^\om$ is the orthogonal of $\G_x$ with respect to $\om$. These relations show that $\pi_x$ is invertible if and only if $\G_x$ is $\om$-nondegenerate and $\rho_x$ is onto.

\paragraph{Left symmetric algebroids} Let $(A,M,\rho)$ be an anchored vector bundle.
A {\it right-anchored product} on  $(A,M,\rho)$ is $\R$-bilinear map $T:\Ga(A)\times\Ga(A)\too\Ga(A)$
such that, for any $a,b\in\Ga(A)$ and any
$f\in C^\infty(M)$,
\begin{equation}\label{eq2t}T_{fa}b=fT_ab\esp T_a(fb)=fT_ab+\rho(a)(f)b.\end{equation}
To $T$ we associate the anchored bracket $[a,b]_T=T_ab-T_ba$.
The curvature of $T$ is the $\R$-trilinear map $R^T:\Ga(A)\times\Ga(A)\times\Ga(A)\too\Ga(A)$ given by
\[ R^T(a,b)c=[T_a,T_b]c-T_{[a,b]_T}c. \]
The curvature $R^T$ satisfies, for any $a,b,c\in\Ga(A)$,
\begin{equation}\label{curvature}\left\{\begin{array}{l} R^T(a,b)c=-R^T(b,a)c,\\
	R^T(fa,b)c=R^T(a,fb)c=fR^T(a,b)c\\
	\;R^T(a,b)fc=fR^T(a,b)c+\tau_{\br_T}(a,b)(f)c,
	\end{array}\right.
	\end{equation}
	and the Bianchi's identity
	\begin{equation}\label{bianchi}
	R^T(a,b)c+R^T(b,c)a+R^T(c,a)b=J_{\br_T}(a,b,c).
	\end{equation}

From theses relations we deduce that $R^T$ is local and $T$ is called Lie-admissible if $\br_T$ induces a Lie algebroid structure on $(A,M,\rho)$. The following proposition follows easily from Proposition \ref{pr1} and \eqref{bianchi}.

\begin{pr}\label{pr2} Let $(A,M,\rho)$ be an anchored bundle and $T$  a right-anchored product on $\Ga(A)$. Then $T$ is Lie-admissible if and only for any $x\in M$ there exists an open set $U$ of $M$ containing $x$ and a basis of sections $(a_1,\ldots,a_r)$ over $U$ such that
	\[ R^T(a_i,a_j)a_k+R^T(a_j,a_k)a_i+R^T(a_k,a_i)a_j=0\esp \tau_{\br_T}(a_i,a_j)=0,\quad 1\leq i<j<k\leq r. \]

\end{pr}

The following proposition is a consequence of \eqref{curvature}.

\begin{pr}\label{pr3} Let $(A,M,\rho)$ is an anchored bundle and $T$ is a right-anchored product on $\Ga(A)$. Then the following assertions are equivalent.
	\begin{enumerate}\item[$(i)$] The curvature $R^T$ vanishes identically.
		\item[$(ii)$] For any $x\in M$ there exists an open set $U$ of $M$ containing $x$ and a basis of sections $(a_1,\ldots,a_r)$ over $U$ such that
		\[ R^T(a_i,a_j)a_k=0\esp \tau_{\br_T}(a_i,a_j)=0,\quad 1\leq i<j\leq r,\; 1\leq k\leq r. \]
		
	\end{enumerate}

\end{pr}

\begin{Def}
	
	A {\it left symmetric algebroid} is an anchored bundle $(A,M,\rho)$ together with a right-anchored product $T$ satisfying $(i)$ or $(ii)$ of Proposition \ref{pr3}. 
	
\end{Def}

It is obvious that if $(A,M,\rho,T)$ is a left symmetric algebroid then $(A,M,\rho,\br_T)$ is a Lie algebroid.

\begin{remark} \begin{enumerate}\item A  right-anchored product is Lie-admissible if the Jacobiator of the associated anchored bracket vanishes. Proposition \ref{pr2} gives a subtle way of checking the Lie-admissibility. We will use it in a crucial way in the study of para-K\"ahlerian Lie algebroids.
		\item A left symmetric algebroid is an anchored bundle with a right-anchored product whose curvature vanishes. However, it is important to keep in mind that the vanishing of the curvature is not a tensorial equation and  $(ii)$ of Proposition \ref{pr3} could be very useful in concrete situations.
		
	\end{enumerate}

\end{remark}

\paragraph{Some examples of left symmetric algebroids}\begin{enumerate}\item Any left symmetric algebra is obviously a left symmetric algebroid.\item Let $M$ be a smooth manifold. A right-anchored product on $(TM,M,\mathrm{id}_{TM})$ is just a linear connection, a Lie-admissible right-anchored product  on $(TM,M,\mathrm{id}_{TM})$ is just a torsion-free linear connection and a left symmetric product on $(TM,M,\mathrm{id}_{TM})$ is just a flat linear connection.
	\item Let $(S,.)$ be a left symmetric algebra, i.e, for any $a,b,c\in S$,
	\[ \mathrm{assoc}(a,b,c)=\mathrm{assoc}(b,a,c), \]where
	$\mathrm{assoc}(a,b,c)=(a.b).c-a.(b.c)$. It is know that $[a,b]=a.b-b.a$ is a Lie bracket on $A$. Let $\tau:S\too\mathcal{X}(M)$ an action of $(S,[\;,\;])$ on a smooth manifold $M$. Consider the anchored bundle $(A,M,\rho)$  where $A$ is the trivial bundle $M\times S$ and $\rho(m,a)=\tau(a)(m)$. Define on $\Ga(A)=C^\infty(M,S)$ the product $T$ by
	\[ T_{s_1}s_2=\mathcal{L}_{\rho(s_1)}s_2+s_1.s_2. \]
	By using $(ii)$ of Proposition \ref{pr3}, it is easy to check that $(A,M,\rho,T)$ is a left symmetric algebroid.	
	
	\item
	Let $\G\stackrel{\Ga}\too\X(M)$ be an action of a finite-dimensional real
	Lie algebra $\G$ on a smooth manifold $M$,
	i.e., a morphism of Lie algebras
	from $\G$ to the Lie algebra of vector fields
	on $M$.
	Let  $r\in\wedge^2\G$  be a solution of the classical Yang-Baxter equation, i.e.,
	$$[r,r]=0,$$where $[r,r]\in\G\wedge\G\wedge\G$ is defined
	by
	$$[r,r](\al,\be,\ga)=\al([r(\be),r(\ga)])+\be([r(\ga),r(\al)])+
	\ga([r(\al),r(\be)]),$$and $r:\G^*\too\G$ denotes also the linear
	map given by $\al(r(\be))=r(\al,\be).$
	We denote by $\pi^r$ the Poisson tensor on $M$ image of  $r$ by $\Ga$.
	Write
	$$r=\sum_{i,j}a_{ij}u_i\wedge u_j$$ and put, for
	$\al,\be\in\Om^1(M)$,
	$$\D^{r}_\al\be:=\sum_{i,j}a_{ij}\al(U_i)\mathcal{L}_{U_j}\be,$$where  $U_i=\Ga(u_i)$.
	We get a map $\D^r:\Om^1(M)\times\Om^1(M)\too\Om^1(M)$ which is a
	a right-anchored product on $(T^*M,M,\pi^r_{\#})$. It was proved in \cite{bouc3} that  
	$\D^r$ is Lie-admissible and left symmetric.

\end{enumerate}

We finish this section by an useful lemma.

\begin{Le}\label{saïd} Let $(A,M,\rho,T)$ a left symmetric algebroid such that $\rho(a)=0$ implies $T_a=0$. Then, for any $m\in M$ a regular point and $a\in A_m$, there exists an open set $U$ around $m$ and $s\in\Ga(U)$ such that $Ts=0$ and $s(m)=a$.
	
\end{Le}

\begin{proof} Denote by $q$ the rank of $\rho$ at $m$. According to the local splitting theorem near a regular point (see \cite{fer1} Theorem 1.1), there exists a coordinates system $(x_1,\ldots,x_q,y_1,\ldots,y_{n-q})$ around $m$ and a basis of sections $(a_1,\ldots,a_r)$ of $A$  such that
	\[ \left\{\begin{array}{lll}
	\rho(a_i)=\partial_{x_i},\;i=1,\ldots,q,\\
	\rho(a_i)=0,\;i=q+1,\ldots,r,\\
	\;[a_i,a_j]_A=\di\sum_{u=q+1}^rC_u^{ij}a_u.\end{array}      \right. \] Put $T_{a_i}a_j=\sum_{u=1}^r\Ga_{ij}^ua_u$. We look for $s=\sum_{j=1}^rf_ja_j$ satisfying $Ts=0$ and $s(m)=\sum_{i=1}^ra_i^0a_i$. Since by assumption $T_{a_i}=0$ for $i=q+1,\ldots,r$, this is equivalent to
	\begin{equation}\label{s} \frac{\partial f_j}{\partial x_i}=-\sum_{u=1}^rf_u\Ga_{iu}^j\esp f_j(m)=a_j^0,\;i=1,\ldots,q,j=1,\ldots,r. \end{equation}
	We thank of the $y_i$ as  parameters and we consider $\al:\R^q\too\R$ and, for $i=1,\ldots,q$, $F_i:\R^q\times\R^r\too\R^r$, given by
	\[ \al(x_1,\ldots,x_q)=(f_1(x,y),\ldots,f_r(x,y))\esp
	F_i(x,z)=-\left(\sum_{u=1}^rz_u\Ga_{iu}^j(x,y)\right)_{j=1}^r. \]
	Thus \eqref{s} is equivalent to
	\[ \frac{\partial \al}{\partial x_i}=F_i(x,\al(x)). \]According to a well-known theorem (see \cite{spivak} pp.187) these system of differential equations has solutions if
	\[ \frac{\partial F_j}{\partial x_i}
	-\frac{\partial F_i}{\partial x_j}+\sum_{u=1}^r
	\frac{\partial F_j}{\partial z_u}F_i^u-
	\sum_{u=1}^r
	\frac{\partial F_i}{\partial z_u}F_j^u=0,\; i,j=1,\ldots,q. \]
	Or, one can check easily by using that $T_{[a_i,a_j]}=0$ that this condition is equivalent to the vanishing of the curvature.
\end{proof}

\section{Para-K\"ahler Lie algebroids}\label{section3}
In this section, we give the definition of a para-K\"ahler Lie algebroid, its basic properties,  a characterization of a such structure and some examples. 

 \paragraph{Definition of para-K\"ahler Lie algebroid and its immediate consequences} Let $(A,M,\rho,\br_A)$ be a Lie algebroid.
 The Nijenhuis torsion of a bundle homomorphism  $H:A\too A$  is given by
 \begin{equation}\label{nui}N_H(a,b):=[Ha,Hb]_A-H[Ha,b]_A-H[a,Hb]_A+H^2[a,b]_A,
 \end{equation}for any $a,b\in\Ga(A)$. The following proposition is a generalization of a well-known fact in differential geometry (see Proposition 4.2 pp. 148 \cite{KN}).
 
 \begin{pr}\label{equi}
 	Let $(A,M,\rho,\prs)$ be a pseudo-Riemannian Lie algebroid and $K:A\too A$  a skew-symmetric bundle endomorphism such that $K^2=\mathrm{id}_A$. Define $\Om$ by $\Om(a,b)=\langle Ka,b\rangle$. Then the following assertions are equivalent:
 	\begin{enumerate}\item[$(i)$] $d_A\Om=0$ and $N_K=0$.
 		\item[$(ii)$] $\na K=0$, where $\na$ is the Levi-Civita $A$-connection associated to $\prs$.
 	\end{enumerate}
 	
 \end{pr}
 
 \begin{proof} We have, for any $a,b,c\in\Ga(A)$,
 	\begin{eqnarray*}
 		d_A\Om(a,b,c) 
 		&=&\langle \na_{a}(K){}b,{}c\rangle+\langle \na_{{}b}(K){}c,{}a\rangle+\langle \na_{{}c}(K){}a,{}b\rangle,\\
 		d_A\Om({}a,K{}b,K{}c)&=&\langle \na_{{}a}(K)K{}b,K{}c\rangle+\langle \na_{K{}b}(K)K{}c,{}a\rangle+\langle
 		\na_{K{}c}(K){}a,K{}b\rangle,\\
 		\langle
 		N_K({}c,K{}b),{}a\rangle&=&\langle\na_{K{}c}(K)(K{}b),{}a\rangle-\langle\na_{{}b}(K)({}c),{}a\rangle+\langle\na_{{}c}
 		(K)({}b),{}a\rangle-\langle\na_{K{}b} (K)(K{}c),{}a\rangle,
 	\end{eqnarray*}where $\na_{{}a}(K){}b=\na_a(K{}b)-K(\na_a{}b)$. These relations show that $(ii)$ implies $(i)$. Moreover,
 	 $K$ is skew-symmetric and hence
 	$$\langle \na_{{}a}(K){}b,{}c\rangle+\langle \na_{a}(K){}c,{}b\rangle=0.$$Thus
 	$$d_A\Om({}a,{}b,{}c)+d_A\Om({}a,K{}b,K{}c)+\langle
 	N_K({}c,K{}b),{}a\rangle=\langle \na_{{}a}(K){}b,{}c\rangle+\langle \na_{{}a}(K)K{}b,K{}c\rangle.$$
 	Now
 	\begin{eqnarray*}
 		\langle \na_{{}a}(K)K{}b,K{}c\rangle&=&\langle\na_a{}b,K{}c\rangle-\langle K\na_aK{}b,K{}c,\rangle\\
 		&=&\langle \na_aK{}b,{}c,\rangle-\langle K\na_a{}b,{}c\rangle=\langle \na_{{}a}(K){}b,{}c\rangle.
 	\end{eqnarray*}
 	Finally,
 	$$2\langle \na_{{}a}(K){}b,{}c\rangle=d_A\Om({}a,{}b,{}c)+d_A\Om({}a,K{}b,K{}c)+\langle
 	N_K({}c,K{}b),{}a\rangle$$and the proposition follows.
 \end{proof}

\begin{Def}\label{def1}
	A \emph{para-K\"ahler Lie algebroid} is a pseudo-Riemannian Lie algebroid $(A,M,\rho,\prs)$ endowed with
	a bundle isomorphism $K:A\too A$ satisfying
	$K^2=\mathrm{Id}_A$, $K$  is skew-symmetric with respect to $\prs$ and $\na K=0$, where $\na$ is the Levi-Civita $A$-connection of $\prs$ .\end{Def}

A para-K\"ahler Lie
algebroid $(A,M,\rho,\prs,K)$  carries a natural  bilinear skew-symmetric nondegenerate form
$\Om_K$ defined by
$\Om_K(a,b)=\langle Ka,b\rangle$. The following proposition is an immediate consequence of  Definition \ref{def1} and Proposition \ref{equi}.
\begin{pr}\label{prk} Let $(A,M,\rho,\prs,K)$ be para-K\"ahler Lie
	algebroid and $\na$ its Levi-Civita $A$-connection. Then:
	\begin{enumerate}\item[$(i)$] $(A,M,\rho,K)$ is a para-complex Lie algebroid, i.e., $K^2=\mathrm{Id}_A$, $N_K=0$ and, for any $x\in M$, $\dim\ker(K+\mathrm{Id}_A)(x)=
		\dim\ker(K-\mathrm{Id}_A)(x)$.
		\item[$(ii)$] $(A,M,\rho,\Om_K)$ is a symplectic Lie algebroid and hence $\pi=\rho(\Pi)$  is a Poisson tensor  on $M$, where $\Pi$ is the inverse of $\Om_K$.
		\item[$(iii)$] $A=A^+\oplus A^{-}$ where $A^+=\ker(K-\mathrm{Id}_{A})$ and
		$A^{-}=\ker(K+\mathrm{Id}_{A})$.
		\item[$(iv)$] $A^+$ and $A^{-}$ are  isotropic with respect to $\prs$ and 
		Lagrangian with 
		respect to $\Om_K$.
		\item[$(v)$] for
		any $a\in\Ga(A)$, $\na_a(\Ga(A^+))\subset \Ga(A^+)$ and $\na_a(\Ga(A^{-}))\subset \Ga(A^{-})$.
	\end{enumerate}\end{pr}

The following proposition is the first important property of para-K\"ahler Lie
algebroids.

\begin{pr}\label{pr11}Let $(A,M,\rho,\prs,K)$ be a para-K\"ahler Lie algebroid then,
	for any $a^+,b^+\in\Ga(A^+)$, $a^-,b^-\in\Ga(A^-)$,
	\[ R(a^+,b^+)=R(a^-,b^-)=0\esp R(a^+,a^-)b^+-R(b^+,a^-)a^+=
	R(a^-,a^+)b^--R(b^-,a^+)a^-=0, \]
	where $R$ is the curvature of the Levi-Civita connection $\na$. In particular, for $\e=\pm$, the restriction of $\na$ to $A^\e$ induces on $(A^\e,M,\rho_{|A^\e})$  a left symmetric  algebroid structure.
\end{pr}
\begin{proof}According to Bianchi's identity \eqref{bianchi}
	\[ R(a^+,b^+)a^-+R(b^+,a^-)a^++R(a^-,a^+)b^+=0. \]
	Since $\na\Ga(A^\e)\subset\Ga(A^\e)$ we get $R(a^+,b^+)a^-\in\Ga(A^-)$ and
	$R(b^+,a^-)a^++R(a^-,a^+)b^+\in\Ga(A^+)$ and hence
	\[ R(a^+,b^+)a^-=R(b^+,a^-)a^++R(a^-,a^+)b^+=0. \]
	Now, for any $c^+\in\Ga(A^+)$,
	\[ \langle R(a^+,b^+)a^-,c^+\rangle=-\langle a^-,R(a^+,b^+)c^+\rangle=0. \]
	So $R(a^+,b^+)=0$. In the same way we can show that $R(a^-,b^-)=0$ and
	$R(b^-,a^+)a^-+R(a^+,a^-)b^-=0$. This completes the proof.
\end{proof}

\paragraph{Para-K\"ahler Lie algebroids as phase spaces}
Let $(A,M,\rho,\prs,K)$ be a para-K\"ahler Lie algebroid. As an anchored bundle,  $(A,M,\rho)=(A^+\oplus A^-,M,\rho^+\oplus\rho^-)$ where $\rho^\e=\rho_{|A^\e}$.  The map $\flat:A^{-}\too (A^{+})^*,$  $a\mapsto a^*,$ where $a^*$ is given by
$\prec a^*,b\succ=\langle a,b\rangle,$  realizes a bundle isomorphism between 
$A^{-}$ and $(A^{+})^*$. According to Proposition \ref{pr11},
the restriction of the Levi-Civita connection $\na$ to $A^+$ say $S$ induces on $(A^+,M,\rho^+)$ a structure of left symmetric algebroid. The same thing happens for $A^-$ and by the identification above we get a left symmetric right-anchored product $T$ on $((A^{+})^*,M,\rho_1)$ where $\rho_1=\rho\circ\flat^{-1}$.
Thus we can identify $(A,M,\rho,\prs,K)$ with the phase space $(\Phi(A^+),\rho^+\oplus\rho_1,\prs_0,,K_0)$ and,  under this identification, the Levi-Civita connection $\na^0$  is entirely determined by $S$ and $T$. Namely, for any $a,b\in\Ga(A^+)$ and $u,v\in\Ga((A^+)^*)$, we have
\[ \na_a^0b=S_ab,\; \na_u^0v=T_uv,\; \na_a^0u=S_a^*u\esp \na_u^0a=T_u^*a, \]where $S^*$ and $T^*$ are the dual of $S$ and $T$ given by \eqref{dual}. 
Moreover, since $\rho^+\oplus\rho_1:\Ga(\Phi(A^+))\too TM$ is a Lie algebra homomorphism, we have
\[[\rho^+(a),\rho^+(b)]=\rho^+(S_ab-S_ba),\;[\rho_1(u),\rho_1(v)]=\rho_1(T_uv-T_vu) \esp[\rho^+(a),\rho_1(u)]=\rho_1(S_a^*u)-\rho^+(T_u^*a), \]
for any $a,b\in\Ga(A^+)$ and $u,v\in\Ga((A^+)^*)$.

\paragraph{How to build a para-K\"ahler Lie algebroid from two left symmetric algebroid structures on two dual vector bundles}
Let $(B,M,\rho_0,S)$ and $(B^*,M,\rho_1,T)$ be two left symmetric algebroid structures.
We extend the right-anchored products on $(B,M,\rho_0)$ and $(B^*,M,\rho_1)$ to
$(\Phi(B),M,\rho_0\oplus\rho_1)$
by putting,
for any $X,Y\in \Ga(B)$ and for any $\al,\be\in \Ga(B^*)$,
\begin{equation}\label{eq16}\na_{X+\al}(Y+\be)=S_XY
+T_\al^*Y+S_X^*\be+T_\al\be.\end{equation} 
The anchored bracket associated to $\na$ is given by
\begin{equation}\label{eq16b}[X+\al,Y+\be]_\phi=[X,Y]_B+[\al,\be]_{B^*}+
T_\al^*Y-T_\be^*X+S_X^*\be-S_Y^*\al.
\end{equation} The question now is under which conditions $\na$ is Lie-admissible or equivalently 
$[\;,\;]_\phi$ is a Lie bracket. It is a crucial step in our study and we will use
Proposition \ref{pr2} to get an answer which will turn out te be very useful, particularly, in the next section.

\begin{pr}\label{pr13}With the hypothesis above, the following assertions are equivalent:
	\begin{enumerate}
		\item[$(i)$] The right-anchored product on $(\Phi(B),M,\rho_0\oplus\rho_1)$ given   by \eqref{eq16} is Lie-admissible.
		\item[$(ii)$] For any $X,Y\in\Ga(B)$ and $\al,\be\in\Ga(B^*)$,
		\begin{equation}\label{eq190}R^\na(X,\al)Y=R^\na(Y,\al)X\esp
		R^\na(\al,X)\be=R^\na(\be,X)\al.
		\end{equation}
		
		\item[$(iii)$] For any $x\in M$ there exists an open set $U$ containing $x$ and a basis of sections $(a_1,\ldots,a_n)$ of $B$ over $U$ such that, for any $1\leq i,j,k\leq n,$
		\begin{equation}\label{eq19}R^\na(a_i,\al_k)a_j=R^\na(a_j,\al_k)a_i,\;
		R^\na(\al_i,a_k)\al_j=R^\na(\al_j,a_k)\al_i\esp [\rho_0(a_i),\rho_1(\al_j)]=\rho_1(S_{a_i}^*\al_j)-\rho_0(T_{\al_j}^*a_i),
		\end{equation}
		where $(\al_1,\ldots,\al_n)$ is the dual basis of $(a_1,\ldots,a_n)$ and $R^\na$ is the curvature of $\na$.\end{enumerate}
\end{pr}
\begin{proof} It is a consequence of Proposition \ref{pr2} and the facts that
	\[ R^\na(X,Y)=0,\; R^\na(\al,\be)=0,\;\tau_{\br_\phi}(X,Y)=\tau_{\br_\phi}(\al,\be)=0\esp
	\tau_{\br_\phi}(X,\al)=\rho_1(S_{X}^*\al)-\rho_0(T_{\al}^*X)-
	[\rho_0(X),\rho_1(\al)] \]
	for any $X,Y\in\Ga(B)$ and $\al,\be\in \Ga(B^*)$.
\end{proof}

\begin{Def}
	Two left symmetric  algebroids  $(B,M,\rho_0,S)$ and $(B^*,M,\rho_1,T)$ satisfying \eqref{eq190} or  \eqref{eq19} will be
	called \emph{Lie-extendible} or compatible.\end{Def}
Thus we get the following result.
\begin{theo}\label{theoextendible}
	Let $(B,M,S,\rho_0)$ and $(B^*,M,T,\rho_1)$ two Lie-extendible left symmetric  algebroids. Then $(\Phi(B),M,\rho_0\oplus\rho_1,\prs_0,K_0)$ endowed with the Lie algebroid bracket  given by \eqref{eq16b} is a para-K\"ahler Lie algebroid. Moreover, all para-K\"ahler Lie algebroids are obtained in this way.
	
\end{theo}

\paragraph{A subtlety of compatible left symmetric algebroids}
The compatibility between two left symmetric  algebroids has a subtle property we will point out now.\\
Let $(B,M,\rho_0,S)$ and $(B^*,M,\rho_1,T)$ be two left symmetric algebroid structures. They are compatible if \eqref{eq190} holds.
Note first that, for any $X,Y\in\Ga(B)$ and $\al,\be\in \Ga(B^*)$, 
\begin{equation}\label{eq17}R^\na(X,\al)Y=[S_X,T_\al^*]Y
+S_{T^*_\al X}Y-T_{S^*_X\al}^*Y\esp
R^\na(\al,X)\be=[T_\al,S_X^*]\be
+T_{S^*_X \al}\be-S_{T^*_\al X}^*\be.
\end{equation}
On the other hand, the equation \eqref{eq190} is not tensorial  and it is equivalent to the vanishing of the Jacobiator of the bracket $[\;,\;]_\phi$ given by \eqref{eq16b}. We have seen that the vanishing of the Jacobiator implies the vanishing of the torsion. Let compute the torsion of $[\;,\;]_\phi$. Since the torsions of $[\;,\;]_S$ and $[\;,\;]_T$ vanish, we have for any $X,Y\in\Ga(B)$, $\al,\be\in\Ga(B^*)$,
\[ \tau_{[\;,\;]_\phi}(X,Y)=0\esp \tau_{[\;,\;]_\phi}(\al,\be)=0. \]Moreover,
\begin{equation}\label{eqtau} \tau_{S,T}(X,\al):=\tau_{[\;,\;]_\phi}(X,\al)=[\rho_0(X),\rho_1(\al)]-\rho_1(S_{X}^*\al)+\rho_0(T_{\al}^*X). \end{equation}
Then $\tau_{[\;,\;]_\phi}=0$ if and only if the tensor field $\tau_{S,T}\in\Ga(B^*\otimes B\otimes TM)$ vanishes. By using Bianchi's identity, we get for any $X,Y\in\Ga(B)$ and $\al,\be\in \Ga(B^*)$,
\[ R^\na(X,\al)Y-R^\na(Y,\al)X=J_{[\;,\;]_\phi}(X,\al,Y)\esp
R^\na(\al,X)\be-R^\na(\be,X)\al=J_{[\;,\;]_\phi}(\al,X,\be). \]
Thus if  $\tau_{S,T}$ vanishes then $\rho_0\oplus\rho_1$ is a Lie algebras homomorphism and  hence $\rho_0\oplus\rho_1(J_{[\;,\;]_\phi})=0$. So  we get from Bianchi's identity that
\begin{equation}\label{eq190bis}\rho_0(R^\na(X,\al)Y)=\rho_0(R^\na(Y,\al)X)\esp
\rho_1(R^\na(\al,X)\be)=\rho_1(R^\na(\be,X)\al).\end{equation}So we get the following proposition.
\begin{pr}\label{pr34} Let $(B,M,\rho_0,S)$ and $(B^*,M,\rho_1,T)$ be two left symmetric algebroid structures. Then the following assertions hold:
	\begin{enumerate}\item[$(i)$] If $\rho_0$ and $\rho_1$ are into then $(B,M,S,\rho_0)$ and $(B^*,M,T,\rho_1)$ are Lie-extendible if and only if $\tau_{S,T}=0$.
		\item[$(ii)$] If $\rho_0$ is  into then the two left symmetric structures are Lie-extendible if and only if, 
		for any $X\in\Ga(B)$ and $\al,\be\in \Ga(B^*)$,
		$$ R^\na(\al,X)\be=R^\na(\be,X)\al\esp \tau_{S,T}=0.  $$
		\item[$(iii)$] If $\rho_0$ is an isomorphism and the two left symmetric structures are Lie-extendible then,
		 for any $X\in\Ga(B)$ and $\al\in \Ga(B^*)$,
		\[ T_{\al}^*X=\mathrm{s}(S_{X}^*\al)-[X,\mathrm{s}(\al)]_B, \]
		where $\mathrm{s}=\rho_0^{-1}\circ\rho_1:B^*\too B$.  
		
	\end{enumerate}

\end{pr}

The general case of $(iii)$ in the proposition above will be studied in the next section devoted to the notion of exact para-K\"ahler Lie algebroids which generalizes the notion of exact para-K\"ahler Lie algebras introduced in \cite{bai} and studied in more details in \cite{bouben}.

 It is important to point out that two compatible left symmetric  algebroids give rise to a symmetric bivector field and a Poisson structure on the underlying manifold. Indeed,  if $(B,M,S,\rho_0)$ and $(B^*,M,T,\rho_1)$ are two Lie-extendible left symmetric  algebroids then $(\Phi(B),M,\rho_0\oplus\rho_1,\Om_0)$ is a para-K\"ahler  Lie algebroid and hence there exists a symmetric bivector field $h$ and a  Poisson tensor $\pi$ on $M$ given by \eqref{metric} and \eqref{poisson}, respectively. One can see easily that $h_\#, \pi_\#:T^*M\too TM$ associated to $h$ and  $\pi$ are given by
 \begin{equation}\label{poisson1}h_\#=\rho_1\circ\rho_0^*+\rho_0\circ\rho_1^*\esp
 \pi_\#=\rho_1\circ\rho_0^*-\rho_0\circ\rho_1^*.
 \end{equation}
 
 \begin{exem}\label{exemple1} 
 	Let $(A,M,S,\rho)$ be a left symmetric algebroid. Then the left symmetric product on $A$ and the trivial left symmetric product on $A^*$ together with the trivial anchor are Lie-extendible so $(\Phi(A),M,\rho\oplus0,\prs_0,K_0)$ endowed with the Lie algebra bracket associated to the left symmetric product
 	\begin{equation}\label{exem1eq1} \na_{X+\al}^0(Y+\be)=S_XY+S_X^*\be \end{equation}is a para-K\"ahler Lie algebra. We denote by $[\;,\;]^\triangleright$ the Lie bracket associated to $\na^0$. We have 
 	\[ [X+\al,Y+\be]^\triangleright=[X,Y]+S_X^*\be-S_Y^*\al. \]
 	Moreover, it is easy to check that $(\Phi(A),[\;,\;]^\triangleright,\prs_0)$ is a flat pseudo-Riemannian  Lie algebroid
 	and $(\Phi(A),M,[\;,\;]^\triangleright,\Om_0)$ is a  symplectic Lie algebroid and $\Om_0$ is parallel with respect to $\na^0$.
 \end{exem}
 
\section{Exact para-K\"ahler Lie algebroids }\label{section4}

In this section, we introduce the notion of exact para-K\"ahler Lie algebroids which generalizes exact para-K\"ahler Lie algebras introduced in \cite{bai} and studied in more details in \cite{bouben}.

Let $(A,M,\rho,S)$ be a left symmetric algebroid,  $\mathrm{r}\in\Ga(A\otimes A)$ and $\mathrm{r}=\mathfrak{a}+\mathfrak{s}$ the decomposition of $\mathrm{r}$ into skew-symmetric and symmetric part. We denote by $\mathrm{r}_\#:A^*\too A$ the bundle homomorphism given by $\be(\mathrm{r}_\#(\al))=\mathrm{r}(\al,\be)$.
Put $\rho_\mathrm{r}=\rho\circ\mathrm{r}_\#$ and, for any $\al,\be\in\Ga(A^*)$ and $X\in\Ga(A)$,
\begin{equation}\label{produitr}
\prec T_\al\be,X\succ:=\rho(X).\mathrm{r}(\al,\be)-\mathrm{r}(S_X^*\al,\be)-\mathrm{r}(\al,S_X^*\be)+\prec S_{\mathrm{r}_\#(\al)}^*\be,X\succ=S_X\mathrm{r}(\al,\be)+\prec S_{\mathrm{r}_\#(\al)}^*\be,X\succ.
\end{equation}
It is clear that $T$ is a right-anchored product on $(A^*,M,\rho_\mathrm{r})$.
Let $\na$ be the extension of $S$ and $T$ on  $(\Phi(A),M,\rho\oplus\rho_\mathrm{r})$ given by \eqref{eq16}.\begin{problem}\label{problem1}
Under which conditions on $(A,M,\rho,\mathrm{r})$ is $(A^*,M,\rho_\mathrm{r},T)$  a left symmetric algebroid with $(A,M,\rho,S)$ and $(A^*,M,\rho_\mathrm{r},T)$ being compatible? \end{problem}
Recall that $(A^*,M,\rho_\mathrm{r},T)$ is  a left symmetric algebroid with $(A,M,\rho,S)$ and $(A^*,M,\rho_\mathrm{r},T)$ being compatible if and only if
 of the curvature of $T$ vanishes and,
	for any  $\al,\be\in\Ga(A^*)$, $X,Y\in\Ga(A)$,  
	$$R^\na(X,\al)Y=R^\na(Y,\al)X\esp R^\na(\al,X)\be=R^\na(\be,X)\al.$$
Remark first that if the curvature of $T$ vanishes  then $T$ is Lie-admissible and hence $\rho_\mathrm{r}:\Ga(A^*)\too\mathcal{X}(M)$ is a Lie algebra homomorphism, i.e., $\rho_\mathrm{r}([\al,\be]_T)=[\rho_\mathrm{r}(\al),\rho_\mathrm{r}(\be)]$. This can be written
\begin{equation}\label{eqcn}
\rho\left(\De(r)(\al,\be)\right)=0,
\end{equation}where
\begin{equation}
\De(r)(\al,\be)={\mathrm{r}_\#([\al,\be]_T)}-[\mathrm{r}_\#(\al),\mathrm{r}_\#(\be)]_S.
\end{equation}	We have $\De(r)\in\Ga(A\otimes A\otimes A)$ and the equation \eqref{eqcn} is tensorial. The following theorem gives an answer to  Problem \ref{problem1}.
	
\begin{theo}\label{theo1}Let $(A,M,\rho,S)$ be a left symmetric algebroid and $\mathrm{r}=
	\mathfrak{a}+\mathfrak{s}\in\Ga(A\otimes A)$. Then $(A^*,M,\rho_{\mathrm{r}},T)$ is a left symmetric algebroid with $(A,M,\rho,S)$ and $(A^*,M,\rho_\mathrm{r},T)$ being compatible if and only if, for any $\al,\be\in\Ga(A^*)$ and $X\in\Ga(A)$,
	\begin{equation}\label{exact}\rho\circ\De(r)=0,\; S^2\mathfrak{a}=0\esp
	Q_X\De(r)(\al,\be):=[X,\De(r)(\al,\be)]_S-\De(r)(S_X^*\al,\be)-\De(r)(\al,S_X^*\be)=0.\end{equation}
	
\end{theo}

Before proving this theorem, let us give some clarifications on \eqref{exact}. First, note that $S\mathfrak{a}$ and $S^2\mathfrak{a}$ are given by
\[ S_X\mathfrak{a}(\al,\be)=\rho(X).\mathfrak{a}(\al,\be)-\mathfrak{a}(S_X^*\al,\be)-\mathfrak{a}(\al,S_X^*\be)\esp S^2_{X,Y}\mathfrak{a}=S_XS_Y\mathfrak{a}-S_{S_XY}\mathfrak{a}. \]
The second point is that the quantity $Q_X\De(r)(\al,\be)$ is tensorial with respect to $\al$ and $\be$, it is not tensorial with respect to $X$. However,  when $\rho\circ\De(r)=0$ then it becomes tensorial with respect to $X$ and hence the last equation in \eqref{exact} is tensorial. The proof of  Theorem \ref{theo1} is a consequence of the following lemma.

\begin{Le}\label{lemma} Let $(A,M,\rho,S)$ be a left symmetric algebroid and $\mathrm{r}=
	\mathfrak{a}+\mathfrak{s}\in\Ga(A\otimes A)$. With the notations above, we have, for any $X,Y\in A$ and for any $\al,\be,\ga\in A^*$
	\[ R^\na(X,\al)Y=R^\na(Y,\al)X,\;\prec R^\na(\al,X)\be-R^\na(\be,X)\al,Y\succ=-2S^2_{X,Y}\mathfrak{a}(\al,\be)\esp \]
	\begin{eqnarray*}
		\prec R^T(\al,\be)\ga,X\succ
		&=&-\rho\left(\De(r)(\al,\be)\right).\prec\ga,X\succ-
		\prec\ga,[X,\De(r)(\al,\be)]_S\succ+\prec\ga,\De(r)(S_X^*\al,\be)\succ
		+\prec\ga,\De(r)(\al,S_X^*\be)\succ\\&&+2S^2_{X,\mathfrak{s}_\#(\ga)}\mathfrak{a}(\al,\be)
		-2S^2_{X,\mathfrak{a}_\#(\ga)}\mathfrak{a}(\al,\be).
	\end{eqnarray*}
\end{Le}

\begin{proof} Note first  that a direct computation using \eqref{produitr} gives, for any $X\in\Ga(A),\al\in\Ga(A^*)$,
	\begin{equation}\label{eqt}
	T_\al^*X=\mathrm{r}_\#(S_X^*\al)+[{\mathrm{r}_\#(\al)},X]_S.
	\end{equation}
	On the other hand, recall from \eqref{eq17} that $R^\na(X,\al)=[S_X,T_\al^*]
	+S_{T^*_\al X}-T_{S^*_X\al}^*$. So by using \eqref{eqt} we get
	\begin{eqnarray*}
		R^\na(X,\al)Y&=&S_X\mathrm{r}_\#(S_Y^*\al)+
		S_X[{\mathrm{r}_\#(\al)},Y]_S-
		\mathrm{r}_\#(S_{S_XY}^*\al)-[{\mathrm{r}_\#(\al)},S_XY]_S\\
		&&+S_{\mathrm{r}_\#(S_X^*\al)}Y+
		S_{[{\mathrm{r}_\#(\al)},X]_S}Y-
		 \mathrm{r}_\#(S_Y^*S^*_X\al)- [{\mathrm{r}_\#(S^*_X\al)},Y]_S\\
		&=&[X,\mathrm{r}_\#(S_Y^*\al)]_S+S_{\mathrm{r}_\#(S_Y^*\al)}X+
		S_X[{\mathrm{r}_\#(\al)},Y]_S-
		\mathrm{r}_\#(S_{S_XY}^*\al)-[{\mathrm{r}_\#(\al)},S_XY]_S\\
		&&+S_{\mathrm{r}_\#(S_X^*\al)}Y+
		S_{[{\mathrm{r}_\#(\al)},X]_S}Y-
		 \mathrm{r}_\#(S_Y^*S^*_X\al)- [{\mathrm{r}_\#(S^*_X\al)},Y]_S\\
		&=&[X,\mathrm{r}_\#(S_Y^*\al)]_S+
		[Y,\mathrm{r}_\#(S_X^*\al)]_S+
		S_{\mathrm{r}_\#(S_X^*\al)}Y+
		S_{\mathrm{r}_\#(S_Y^*\al)}X+
		[X,[{\mathrm{r}_\#(\al)},Y]_S]_S\\&&-
		\mathrm{r}_\#(S_{S_XY}^*\al+S_Y^*S^*_X\al)-
		[{\mathrm{r}_\#(\al)},S_XY]_S
		+
		S_{[{\mathrm{r}_\#(\al)},X]_S}Y+
		S_{[{\mathrm{r}_\#(\al)},Y]_S}X.
	\end{eqnarray*}
	So
	\begin{eqnarray*} 
		R^\na(X,\al)Y-R^\na(Y,\al)X&=&\mathrm{r}_\#((R^S(Y,X))^*\al)+[\mathrm{r}_\#(\al),[Y,X]_S]_S+[X,[\mathrm{r}_\#(\al),Y]_S]_S+[Y,[X,\mathrm{r}_\#(\al)]_S]_S=0. 
	\end{eqnarray*}This shows the first relation.

	We have from \eqref{eq17} that
	$ R^\na(\al,X)=[{T}_\al,{S}_X^*]
	+{T}_{S^*_X \al}-{S}_{{T}^*_\al X}^*.$ By using \eqref{produitr} and \eqref{eqt}, we get
	\begin{eqnarray*}
		\prec R^\na(\al,X)\be,Y\succ&=&\prec T_\al{S}_X^*\be,Y\succ -\prec {S}_X^*{T}_\al\be,Y\succ+\prec{T}_{S^*_X \al}\be,Y\succ-\prec{S}_{{T}^*_\al X}^*\be,Y\succ\\
		&=&S_Y\mathrm{r}(\al,{S}_X^*\be)+\prec S_{\mathrm{r}_\#(\al)}^*{S}_X^*\be,Y\succ
		-\rho(X).\prec {T}_\al\be,Y\succ+\prec {T}_\al\be,{S}_XY\succ\\
		&&+S_Y\mathrm{r}(S^*_X\al,\be)+\prec S_{\mathrm{r}_\#(S^*_X\al)}^*\be,Y\succ
		-\prec{S}_{\mathrm{r}_\#(S_X^*\al)}^*\be,Y\succ-\prec{S}_{[\mathrm{r}_\#(\al),X]_S}^*\be,Y\succ\\
		&=&S_Y\mathrm{r}(\al,{S}_X^*\be)+S_Y\mathrm{r}(S^*_X\al,\be)+\prec S_{\mathrm{r}_\#(\al)}^*{S}_X^*\be,Y\succ
		-\rho(X).S_Y\mathrm{r}(\al,\be)-\rho(X).\prec S_{\mathrm{r}_\#(\al)}^*\be,Y\succ\\
		&&+S_{S_XY}\mathrm{r}(\al,\be)+\prec S_{\mathrm{r}_\#(\al)}^*\be,S_XY\succ
		-\prec{S}_{[\mathrm{r}_\#(\al),X]_S}^*\be,Y\succ\\
		&=&-\rho(X).S_Y\mathrm{r}(\al,\be)+S_Y\mathrm{r}(\al,{S}_X^*\be)+S_Y\mathrm{r}
		(S^*_X\al,\be)
		+S_{S_XY}\mathrm{r}(\al,\be)\\
		&=&-S_XS_Y\mathrm{r}(\al,\be)+S_{S_XY}\mathrm{r}(\al,\be),
	\end{eqnarray*}and the second relation follows from $\mathrm{r}=\mathfrak{s}+\mathfrak{a}$.

Let us  compute the curvature of $T$. Remark first that from \eqref{produitr} we can derive easily that
\begin{equation}\label{eqbracketr} \prec [\al,\be]_T,X\succ:=\prec T_\al\be-T_\be\al,X\succ=
\prec S_{\mathrm{r}_\#(\al)}^*\be-S_{\mathrm{r}_\#(\be)}^*\al,X\succ+2S_X\mathfrak{a}(\al,\be). \end{equation}
By using \eqref{eqt} once more, we get
\begin{eqnarray*}
\prec T_{\al}T_{\be}\ga,X\succ&=&\rho_\mathrm{r}(\al).\prec T_{\be}\ga,X\succ-\prec T_{\be}\ga,T_{\al}^*X\succ\\
&=&\rho_\mathrm{r}(\al)\circ\rho_\mathrm{r}(\be).\prec \ga,X\succ-\rho_\mathrm{r}(\al).\prec \ga,T_{\be}^*X\succ
-\rho_\mathrm{r}(\be).\prec \ga,T_{\al}^*X\succ+\prec \ga,T_{\be}^*T_{\al}^*X\succ\\
&=&\rho_\mathrm{r}(\al)\circ\rho_\mathrm{r}(\be).\prec \ga,X\succ-\rho_\mathrm{r}(\al).\prec \ga,T_{\be}^*X\succ
-\rho_\mathrm{r}(\be).\prec \ga,T_{\al}^*X\succ\\
&&+\prec \ga,\mathrm{r}_\#(S_{\mathrm{r}_\#(S_X^*\al)}^*\be)\succ
+\prec \ga,[{\mathrm{r}_\#(\be)},\mathrm{r}_\#(S_X^*\al)]\succ
+\prec \ga,\mathrm{r}_\#(S_{[{\mathrm{r}_\#(\al)},X]_S}^*\be)\succ
+\prec \ga,[{\mathrm{r}_\#(\be)},[{\mathrm{r}_\#(\al)},X]_S]_S\succ,\\
&=&\rho_\mathrm{r}(\al)\circ\rho_\mathrm{r}(\be).\prec \ga,X\succ-\rho_\mathrm{r}(\al).\prec \ga,T_{\be}^*X\succ
-\rho_\mathrm{r}(\be).\prec \ga,T_{\al}^*X\succ
+\prec \ga,\mathrm{r}_\#(S_{\mathrm{r}_\#(S_X^*\al)}^*\be)\succ\\&&
+\prec \ga,\De(\mathrm{r})(S_X^*\al,\be)\succ+\prec \ga,\mathrm{r}_\#([\be,S_X^*\al]_T)\succ
+\prec \ga,\mathrm{r}_\#(S_{[{\mathrm{r}_\#(\al)},X]_S}^*\be)\succ
+\prec \ga,[{\mathrm{r}_\#(\be)},[{\mathrm{r}_\#(\al)},X]_S]_S\succ,\\
	\prec T_{[\al,\be]_T}\ga,X\succ&=&\rho_\mathrm{r}([\al,\be]_T).\prec\ga,X\succ-\prec\ga,T_{[\al,\be]_T}^*X\succ\\
	&=&\rho_\mathrm{r}([\al,\be]_T).\prec\ga,X\succ-\prec\ga,\mathrm{r}_\#(S_X^*[\al,\be]_T)\succ-
	\prec\ga,[{\mathrm{r}_\#([\al,\be]_T)},X]_S\succ.
	\end{eqnarray*}
By using the Jacobi identity for $X,{\mathrm{r}_\#(\al)},{\mathrm{r}_\#(\be)}$, we get
\begin{eqnarray*}
	\prec R^T(\al,\be)\ga,X\succ&=&-\rho\left(\De(r)(\al,\be)\right).\prec\ga,X\succ-
	\prec\ga,[X,\De(r)(\al,\be)]_S\succ+\prec\ga,\De(r)(S_X^*\al,\be)\succ
	+\prec\ga,\De(r)(\al,S_X^*\be)\succ\\&&+\prec Q,\mathfrak{s}(\ga)\succ-\prec Q,\mathfrak{a}(\ga)\succ, 
\end{eqnarray*}where
\[ Q=S_{\mathrm{r}_\#(S_X^*\al)}^*\be-S_{\mathrm{r}_\#(S_X^*\be)}^*\al+[\be,S_X^*\al]_T-[\al,S_X^*\be]_T+S_{[{\mathrm{r}_\#(\al)},X]_S}^*\be-S_{[{\mathrm{r}_\#(\be)},X]_S}^*\al+S_X^*[\al,\be]_T. \]
Now, by using \eqref{eqbracketr} and the fact that the curvature of $S$ vanishes, we get
\begin{eqnarray*}
\prec Q,Y\succ&=&\prec S_{\mathrm{r}_\#(\be)}^*S_X^*\al,Y\succ-\prec S_{\mathrm{r}_\#(\al)}^*S_X^*\be,Y\succ+2S_Y\mathfrak{a}(\be,S_X^*\al)-2S_Y\mathfrak{a}(\al,S_X^*\be)+\prec S_{[{\mathrm{r}_\#(\al)},X]_S}^*\be,Y\succ\\&&-\prec S_{[{\mathrm{r}_\#(\be)},X]_S}^*\al,Y\succ
+\rho(X).\prec[\al,\be]_T,Y\succ-\prec[\al,\be]_T,S_XY\succ\\
&=&-2S_Y\mathfrak{a}(S_X^*\al,\be)-2S_Y\mathfrak{a}(\al,S_X^*\be)+\prec S_X^*S_{\mathrm{r}_\#(\be)}^*\al,Y\succ-\prec S_X^*S_{\mathrm{r}_\#(\al)}^*\be,Y\succ+\rho(X).\prec S_{\mathrm{r}_\#(\al)}^*\be,Y\succ\\&&-
\rho(X).\prec S_{\mathrm{r}_\#(\be)}^*\al,Y\succ+2\rho(X).S_Y\mathfrak{a}(\al,\be)-
\prec S_{\mathrm{r}_\#(\al)}^*\be,S_XY\succ+\prec S_{\mathrm{r}_\#(\be)}^*\al,S_XY\succ-2S_{S_XY}\mathfrak{a}(\al,\be)\\
&=&2S^2_{X,Y}\mathfrak{a}(\al,\be).
\end{eqnarray*}

So we get the  lemma.\end{proof}

\begin{co} \label{co1}Let $(A,M,\rho,S)$ be a left symmetric algebroid such that $\rho$ is into and $\mathrm{r}=
	\mathfrak{a}+\mathfrak{s}\in\Ga(A\otimes A)$. Then $(A^*,M,\rho_{\mathrm{r}},T)$ is a left symmetric algebroid with $(A,M,\rho,S)$ and $(A^*,M,\rho_\mathrm{r},T)$ being compatible if and only if
	$$\De(r)=0\esp S^2\mathfrak{a}=0.$$
	
\end{co}

Let $(A,M,\rho,S)$ be a left symmetric algebroid  and $\mathrm{r}=
\mathfrak{a}+\mathfrak{s}\in\Ga(A\otimes A)$ satisfying \eqref{exact}. Then $(A^*,M,\rho_{\mathrm{r}},T)$ is a left symmetric algebroid compatible with $(A,M,\rho,S)$. According to Theorem \ref{theoextendible}, $(\Phi(A),M,\rho\oplus\rho_{\mathrm{r}})$ carries a para-K\"ahler Lie algebroid structure which will be  called exact. This induces on $M$ a symmetric bivector and a Poisson tensor which, by virtue of \eqref{poisson1}, are given by
\begin{equation}\label{relation} h(\al,\be)=
2\mathfrak{s}(\rho^*(\al),\rho^*(\be))\esp \pi(\al,\be)=
2\mathfrak{a}(\rho^*(\al),\rho^*(\be)). \end{equation}

\begin{exem} Let $(A,M,\rho,S)$ be a left symmetric algebroid  and $\mathrm{r}=
	\mathfrak{a}+\mathfrak{s}\in\Ga(A\otimes A)$ which is $S$-parallel, i.e., $S\mathrm{r}=0$. Then $S\mathfrak{a}=0$ and it is easy to check that $\De(\mathrm{r})=0$. Thus $\mathrm{r}$ satisfies \eqref{exact}.
	
\end{exem}

\section{Para-K\"ahler Lie algebroids associated to quasi $\S$-matrices}\label{section5}

In this section, we study a class of exact para-K\"ahler Lie algebroids associated to a kind of solutions of \eqref{exact} we will call quasi $S$-matrices using the same terminology used in the context of para-K\"ahler Lie algebras in \cite{bouben}.

\begin{Def}\label{def}
 A quasi $\S$-matrix of a left symmetric algebroid $(A,M,\rho,S)$ is a $\mathrm{r}=\mathfrak{a}+\mathfrak{s}\in\Ga(A\otimes A)$ such that, for any $\al,\be\in\Ga(A^*)$ and $X\in\Ga(A)$,
\[ \rho\circ\De(r)=0,\; S\mathfrak{a}=0\esp
Q_X\De(r)(\al,\be):=[X,\De(r)(\al,\be)]-\De(r)(S_X^*\al,\be)-\De(r)(\al,S_X^*\be)=0. \]\end{Def}
In what follows, we  focus our attention on the para-K\"ahler Lie algebroid structure on $\Phi(A)$ associated to a quasi $\S$-matrix. We show that the Lie algebroid structure can be described in a precise and simple way. Indeed, let $\mathrm{r}$  be a quasi $S$-matrix.
Then, according to Theorem  \ref{theo1},
the right-anchored product $T$ on $A^*$ given by \eqref{produitr} is left symmetric and
$(\Phi(A),[\;,\;]^r,\rho+\rho_{\mathrm{r}},\prs_0,K_0)$ is a para-K\"ahler Lie algebroid, where
\[ [X+\al,Y+\be]^r=[X,Y]_S+S_{X}^*\be+T_{\al}^*Y-
S_{Y}^*\al-T_{\be}^*X+[\al,\be]_T. \]  We have shown in Example \ref{exemple1} that $\Phi(A)$ carries a left symmetric product $\na^0$ and its associated Lie bracket $[\;,\;]^\triangleright$ induces on $\Phi(A)$ a para-K\"ahler Lie algebroid structure.
We  define  a new bracket on $\Phi(A)$ by putting
\begin{equation}\label{eqnb} [X+\al,Y+\be]^{\triangleright,r}= [X+\al,Y+\be]^\triangleright+\De(\mathrm{r})(\al,\be).\end{equation}

\begin{pr}\label{pralg} $(\Phi(A),[\;,\;]^{\triangleright,r},\rho+0)$ is a Lie algebroid and
	the linear map $\xi:(\Phi(A),[\;,\;]^{\triangleright,r},\rho+0)\too (\Phi(A),[\;,\;]^r,\rho+\rho_{\mathrm{r}})$, $X+\al\mapsto X-\mathrm{r}_\#(\al)+\al$ is an isomorphism of Lie algebroids.
\end{pr}

\begin{proof}  Clearly $\xi$ is bijective. Let us show that $\xi$ preserves the Lie brackets. 
	It is clear that, for any $X,Y\in \Ga(A)$, $\xi\left([X,Y]^{\triangleright,r} \right)=[\xi(X),\xi(Y)]^r$. Now, for any $X\in \Ga(A)$, $\al\in \Ga(A^*)$,
	\begin{eqnarray*}
		\xi\left([X,\al]^{\triangleright,r} \right)&=&\xi(S_X^*\al)\\
		&=&-\mathrm{r}_\#(S_X^*\al)+S_X^*\al\\
		&\stackrel{\eqref{eqt}}=&-T_\al^*X-[X,\mathrm{r}_\#(\al)]_S+S_X^*\al\\
		&=&[X,-\mathrm{r}_\#(\al)+\al]^r\\
		&=&[\xi(X),\xi(\al)]^r.
	\end{eqnarray*}On the other hand, for any $\al,\be\in \Ga(A^*)$,
	\begin{eqnarray*}
		\xi\left([\al,\be]^{\triangleright,r} \right)&=&\xi(\De(\mathrm{r})(\al,\be))\\
		&=&\De(\mathrm{r})(\al,\be),\\
		\;[\xi(\al),\xi(\be)]^r&=&[-\mathrm{r}_\#(\al)+\al,-\mathrm{r}_\#(\be)+\be]^r\\
		&=&[\mathrm{r}_\#(\al),\mathrm{r}_\#(\be)]_S+[\al,\be]_T-S_{\mathrm{r}_\#(\al)}^*\be+
		S_{\mathrm{r}_\#(\be)}^*\al-T_{\al}^*\mathrm{r}_\#(\be)+
		T_{\be}^*\mathrm{r}_\#(\al)\\
		&\stackrel{\eqref{eqt}}=&[\mathrm{r}_\#(\al),\mathrm{r}_\#(\be)]_S-
		\mathrm{r}_\#(S_{\mathrm{r}_\#(\be)}^*\al)+
		\mathrm{r}_\#(S_{\mathrm{r}_\#(\al)}^*\be)+[\mathrm{r}_\#(\be),\mathrm{r}_\#(\al)]_S-[\mathrm{r}_\#(\al),\mathrm{r}_\#(\be)]_S\\
		&=&\mathrm{r}_\#([\al,\be]_T)-[\mathrm{r}_\#(\al),\mathrm{r}_\#(\be)]_S\\
		&=&\De(\mathrm{r})(\al,\be).\qedhere
	\end{eqnarray*}
\end{proof}

We can now transport the para-K\"ahler structure associated to $\mathrm{r}$ from $(\Phi(A),[\;,\;]^r,\prs_0,K_0)$ to $\Phi(A)$ via $\xi$ and we get the following proposition.
\begin{pr}\label{prm}  Let $(A,M,\rho,S)$ be a left symmetric algebroid and $\mathrm{r}=\mathfrak{a}+\mathfrak{s}\in \Ga(A\otimes A)$ a quasi $\S$-matrix. Then $(\Phi(A),[\;,\;]^{\triangleright,r},\rho+0,\prs_r,K_r)$ is a para-K\"ahler Lie algebroid, where
	\[ \langle X+\al,Y+\be\rangle_r=\prec\al,Y\succ+\prec\be,X\succ
	-2\mathfrak{s}(\al,\be)\esp K_r(X+\al)=X-\al-2\mathrm{r}_\#(\al). \]
	
\end{pr}

\section{Symmetric quasi $\S$-matrices on  affine manifolds and generalized pseudo-Hessian structures}\label{section6}

In this section, we study symmetric quasi $\S$-matrices on the left symmetric algebroid $(TM,M,\mathrm{Id}_{TM},\na)$ associated to an affine manifold $(M,\na)$. This leads naturally to a new structure we call generalized pseudo-Hessian structure. There are many similarities between Poisson manifolds as a generalization of symplectic manifolds and generalized pseudo-Hessian manifolds as a generalization of pseudo-Hessian manifolds and we show some of these similarities.

\paragraph{Symmetric quasi $\S$-matrices on  affine manifolds}
Let $(M,\na)$ be an affine manifold, i.e., a manifold endowed with a torsionless flat connection. Then  $(TM,M,Id_{TM},\na)$ is a left symmetric algebroid and according to Definition \ref{def}, a symmetric quasi $\S$-matrix on $(TM,M,Id_{TM},\na)$ is    a symmetric bivector field $h$ on $M$ such that $\De(h)=0$. Let's study this equation more carefully.
We denote by $\D$ the right-anchored product on $(T^*M,M,h_{\#})$ associated to $h$. According to \eqref{produitr}, we have for any $\al,\be\in\Om^1(M)$ and $X\in\mathrm{X}(M)$,
\begin{equation}\label{connectionh} \prec\D_\al\be,X\succ=\na_X h(\al,\be)+\prec\na^*_{h_\#(\al)}\be,X\succ\esp [\al,\be]_\D=
\na^*_{h_\#(\al)}\be-\na^*_{h_\#(\be)}\al. \end{equation}
\begin{pr}
We have, 
for any $\al,\be,\ga\in\Om^1(M)$,\begin{equation}\label{eqdelta}
\prec\ga,\De(h)(\al,\be)\succ=\na_{h_\#(\be)}h(\al,\ga)-\na_{h_\#(\al)}h(\be,\ga),
\end{equation}and
\begin{equation}\label{eqnabla} \prec\ga,h_\#\left(\D_\al\be \right)\succ=\prec\ga,\na_{h_\#(\al)}h_\#(\be)\succ+\prec\be,\De(h)(\al,\ga)\succ. \end{equation}
 \end{pr}\begin{proof} Let's compute
\begin{eqnarray*}
	\prec\ga,\De(h)(\al,\be)\succ&=&\prec\ga,
	h_\#([\al,\be]_{\D})\succ-\prec\ga,[h_\#(\al),h_\#(\be)]\succ\\
	&=&h(\ga,\na^*_{h_\#(\al)}\be)-h(\ga,\na^*_{h_\#(\be)}\al)-\prec\ga,[h_\#(\al),h_\#(\be)]\succ\\
	&=&\na_{h_\#(\be)}h(\al,\ga)-\na_{h_\#(\al)}h(\be,\ga)+h_\#(\al).h(\be,\ga)
	-h_\#(\be).h(\al,\ga)+h(\al,\na^*_{h_\#(\be)}\ga)-h(\be,\na^*_{h_\#(\al)}\ga)
	\\&&-\prec\ga,[h_\#(\al),h_\#(\be)]\succ\\
	&=& \na_{h_\#(\be)}h(\al,\ga)-\na_{h_\#(\al)}h(\be,\ga)+h_\#(\al).h(\be,\ga)
	-h_\#(\be).h(\al,\ga)+\prec \na^*_{h_\#(\be)}\ga,h_{\#}(\al)\succ-\prec \na^*_{h_\#(\al)}\ga,h_{\#}(\be)\succ   \\
	&&-\prec\ga,\na_{h_\#(\al)}h_\#(\be)\succ+\prec\ga,\na_{h_\#(\be)}h_\#(\al)\succ\\
	&=&\na_{h_\#(\be)}h(\al,\ga)-\na_{h_\#(\al)}h(\be,\ga).
\end{eqnarray*}Let's pursue 
\begin{eqnarray*}
	\prec\D_\al\be, h_\#(\ga)\succ&=&\na_{h_\#(\ga)}h(\al,\be)+h(\na_{h_\#(\al)}^*\be,\ga)\\
	&\stackrel{\eqref{eqdelta}}=&\na_{h_\#(\al)}h(\ga,\be)+h(\na_{h_\#(\al)}^*\be,\ga)+\prec\be,\De(h)(\al,\ga)\succ\\
	&=&h_\#(\al).h(\be,\ga)-h(\na_{h_\#(\al)}^*\ga,\be)+\prec\be,\De(h)(\al,\ga)\succ\\
	&=&\prec\ga,\na_{h_\#(\al)}h_\#(\be)\succ+\prec\be,\De(h)(\al,\ga)\succ.
\end{eqnarray*} \end{proof}
By using  Theorem \ref{theo1}, equation \eqref{eqdelta} and Proposition \ref{prm} we get the following theorem.

\begin{theo}\label{new}  Let $(M,\na)$ be an affine manifold and $(TM,M,Id_{TM},\na)$ its associated left symmetric algebroid. Let $h$ be a symmetric bivector field on $M$ and consider $\D$ the right anchored product given by \eqref{connectionh}. 
	Then $(T^*M,M,h_\#,\D)$ is a left symmetric  algebroid compatible with $(TM,M,Id_{TM},\na)$ if and only if, for any $\al,\be,\ga\in\Om^1(M)$,
	\begin{equation}\label{codazi} \na_{h_\#(\al)}h(\be,\ga)-
	\na_{h_\#(\be)}h(\al,\ga)=0. \end{equation}In this case,  $(TM\oplus T^*M,M,[\;,\;]^{\triangleright},Id_{TM}+0,\prs_h,K_h)$ is a para-K\"ahler Lie algebroid, where
	\[ [X+\al,Y+\be]^\triangleright=[X,Y]+\na_X^*\be-\na_Y^*\al,\; \langle X+\al,Y+\be\rangle_h=\prec\al,Y\succ+\prec\be,X\succ
	-2h(\al,\be)\esp K_h(X+\al)=X-2h_\#(\al)-\al. \]Moreover, the associated symplectic form $\Om_0$ is given by
	\[ \Om_0(X+\al,Y+\be)=\prec\be,X\succ-\prec\al,Y\succ. \]

\end{theo}

\begin{remark} This theorem deserves some comments. Indeed, the theorem asserts that, given   an affine manifold $(M,\na)$  and  a symmetric bivector field $h$ satisfying \eqref{codazi}, we have:
	\begin{enumerate}
		\item $(T^*M,M,h_\#,\D)$ is a left symmetric  algebroid,
		\item $(TM\oplus T^*M,M,[\;,\;]^{\triangleright},Id_{TM}+0,\prs_h,\Om_0,K_h)$ is a para-K\"ahler Lie algebroid.
	\end{enumerate} These two results are a	consequence of  a long path based on all the results and the constructions performed before. The first assertion introduces a  class of left symmetric algebroids and hence a class of Lie algebroids which, to our knowledge, has not been considered before. We will devote the reminder of this section to the study of this class. Also, to our knowledge, the class of para-K\"ahler Lie algebroids introduced in 2. has not been considered before. \\
	Now once the results are available, one can prove the assertions 1. and 2. directly. 
	For the first assertion, one can compute (a huge computation identical to the one in Lemma \ref{lemma}) the curvature of $\D$ and show that it vanishes. For the  second assertion, we can use Proposition \ref{equi} and show either
	 that the Nijenhuis torsion of $K_h$ vanishes and $\Om_0$ is closed with respect to $[\;,\;]^{\triangleright}$ or show  that $K_h$ is parallel with respect to the Levi-Civita connection. We have seen in Example \ref{exemple1} that $\Om_0$ is parallel with respect to the Lie-admissible connection $\na^0_{X+\al}(Y+\be)=\na_{X}Y+\na_{X}^*\be$ and hence it is closed. To show that the Nijenhuis torsion with respect to $[\;,\;]^{\triangleright}$ vanishes is an easy computation using that $\De(h)=0$.
	 However, one must point out that $\na^0$ is not the Levi-Civita connection $\overline{\na}$ of $(TM\oplus T^*M,M,[\;,\;]^{\triangleright},Id_{TM}+0,\prs_h)$ and a straightforward computation gives
	 \[ \overline{\na}_XY=\na_XY,\;\overline{\na}_X\al=\na_X^*\al-(\na_Xh)_\#(\al),\;
	 \overline{\na}_\al X=-(\na_Xh)_\#(\al)\esp
	 \overline{\na}_\al\be=-2(\na_{h_\#(\al)}h)_\#(\be)+\na_{h_\#(\al)}^*\be-\D_\al\be.  \]
	 With this formula one can check that $\overline{\na}K_h=0$.
	
\end{remark}

We will give now  other characterizations of bivector fields satisfying \eqref{codazi}.
\begin{pr}\label{hamilton} Let $(M,\na)$ a manifold endowed with a torsionless connection (we don't need to suppose that $\na$ is flat). Le $h$ be a symmetric bivector field on $M$. For any $f\in C^\infty(M)$, put $X_f=h_\#(df)$. Then the following assertions are equivalent.
\begin{enumerate}\item[$(i)$] $h$ satisfies \eqref{codazi}.  
	\item[$(ii)$] For any $f,g\in C^\infty(M)$ and any $\al\in\Om^1(M)$, $d\al(X_f,X_g)=\na_{X_f}h_\#(\al)(g) -\na_{X_g}h_\#(\al)(f)$.
	\item[$(iii)$] For any $f,g,\mu\in C^\infty(M)$, $\na_{X_f}X_\mu(g)= \na_{X_g}X_\mu(f)$.
	\item[$(iv)$]
	For any $x\in M$, there exists a coordinates system $(x_1,\ldots,x_n)$ around $x$ such that for any $1\leq k\leq n$ and $1\leq i<j\leq n$, $\na_{X_{x_i}}X_{x_k}(x_j)= \na_{X_{x_j}}X_{x_k}(x_i)$.
\end{enumerate}	
	
\end{pr}

\begin{proof} We have
	\begin{eqnarray*}
		\na_{X_f}h(dg,\al)&=&X_f.h(dg,\al)-\prec\na_{X_f}^*dg,h_\#(\al)\succ-\prec\na_{X_f}^*\al,X_g\succ\\
		&=&X_f.h(dg,\al)-X_f.h(dg,\al)+\na_{X_f}h_\#(\al)(g)-X_f.h(dg,\al)+\al(\na_{X_f}X_g)\\
	&=&	-X_f.\al(X_g)+\na_{X_f}h_\#(\al)(g)+\al(\na_{X_f}X_g).
	\end{eqnarray*}Thus
	\[\na_{X_f}h(dg,\al)-\na_{X_g}h(df,\al)= -d\al(X_f,X_g)+\na_{X_f}h_\#(\al)(g) -\na_{X_g}h_\#(\al)(f). \]
This relation and the fact that \eqref{codazi} is tensorial permit to prove the proposition.
\end{proof}

\begin{remark}  Let $(M,\na)$ as in Proposition \ref{hamilton} and $h$ satisfying \eqref{codazi}. By using $(iii)$ of Proposition \ref{hamilton}, we get for any $f,g,\mu\in C^\infty(M)$,
\begin{eqnarray*}
\;[X_f,X_g](\mu)&=&\na_{X_f}X_g(\mu)-\na_{X_g}X_f(\mu)\\
&=&\na_{X_\mu}X_g(f)-\na_{X_\mu}X_f(g)\\
&=&[X_\mu,X_g](f)+[X_f,X_\mu](g)+\na_{X_g}X_\mu(f)-\na_{X_f}X_\mu(g)\\
&=&[X_\mu,X_g](f)+[X_f,X_\mu](g).
\end{eqnarray*}	Thus
\begin{equation}\label{jacobi}
[X_f,X_g](\mu)+[X_g,X_\mu](f)+[X_\mu,X_f](g)=0.
\end{equation}So the triple product $\{.,.,.\}:C^\infty(M)\times C^\infty(M)\times C^\infty(M)\too C^\infty(M)$ given by $\{f,g,\mu\}=[X_f,X_g](\mu)$ satisfies:
\begin{enumerate}
	\item $\{f,g,\mu\}=-\{g,f,\mu\}$,
	\item $\{f,g,\mu\}+\{g,\mu,f\}+\{\mu,f,g\}=0$,
	\item $\{f,g,\mu_1\mu_2\}=\{f,g,\mu_1\}\mu_2+\{f,g,\mu_2\}\mu_1$.
\end{enumerate}
Note that we have a similar situation on a Poisson manifold. Indeed, if $M$ is a manifold endowed with a Poisson bracket $\{\;,\;\}$ then the triple product $\langle\;,\;,\;\rangle$ given by
$\langle f,g,\mu\rangle=\{\{f,g\},\mu  \}$ satisfies the relations 1.,2.,3. above.

\end{remark}

\paragraph{Generalized pseudo-Hessian manifolds}
We will show now that the triple $(M,\na,h)$ satisfying \eqref{codazi} are a generalization of a well-known structure, namely, a pseudo-Hessian structure.  Recall that a pseudo-Hessian manifold  (see \cite{shima}) is a triple $(M,\na,g)$ where $\na$ is a
flat torsionless connection  and $g$ is a pseudo-Riemannian metric  is given locally by $g=\na d\phi$ where $\phi$ is a local function. This is equivalent to $S:=\na g$ is totally symmetric, i.e., $(\na,g)$  satisfying the Codazzi equation
\begin{equation}\label{codazzi}\na_Xg(Y,Z)=\na_Yg(X,Z).\end{equation}
If we put $h=g^{-1}$ and take $X=X_v$, $Y=X_v$ and $Z=X_w$ with $u,v,w\in C^\infty(M)$, on can see easily that this equation is equivalent to $\na_{X_u}X_w(v)=\na_{X_v}X_w(u)$ and hence, by virtue of Proposition \ref{hamilton}, $g$ satisfies Codazzi equation if and only if $g^{-1}$ satisfies \eqref{codazi}. There is a subclass of the class of pseudo-Hessian manifolds, namely, the subclass of affine special real manifolds which appeared in physics. A pseudo-Hessian manifold $(M,\na,g)$ is called affine special real manifold if, in addition, $S$ is parallel. In \cite{cortes}, this subclass has been studied in detail and, in particular, the $r$-map which associate to any pseudo-Hessian manifold $(M,\na,g)$ a natural pseudo-K\"ahlerian structure on $TM$ has been scrutinized. By virtue of what above, the following definition is natural.

\begin{Def} We call a triple $(M,\na,h)$ where $\na$ is torsionless flat and $h$  satisfying \eqref{codazi} a generalized pseudo-Hessian manifold. If, in addition, the tensor field $T$ given by $T(\al,\be,\ga)=\na_{h_\#(\al)}h(\be,\ga)$ is parallel with respect to the anchored product $\D$ given by \eqref{connectionh}, we call $(M,\na,h)$ generalized affine special real manifold.
	
\end{Def}
The following theorem shows a similarity between Poisson manifolds and pseudo-Hessian manifolds.
 
\begin{theo}\label{poissonbis} Let $(M,\na,h)$ be generalized pseudo-Hessian  manifold. Then $\mathrm{Im}h_\#$ is integrable and defines a singular foliation on $M$ such that for every leaf $L$ we have:
	\begin{enumerate}
		\item[$(i)$] For every vector fields $X,Y$ tangent to $L$, $\na_XY$ is tangent to $L$,
		\item[$(ii)$] $L$ has a natural  pseudo-Hessian structure. Moreover, if $(M,\na,h)$ is a generalized affine special real manifold then $L$ is an affine special real manifold.
		
	\end{enumerate}
	
\end{theo}

\begin{proof}  According to Theorem \ref{new}, $(T^*M,M,h_\#,\D)$ is a left symmetric Lie algebroid and hence $(T^*M,M,h_\#,\br_\D)$ is a Lie algebroid. This implies that $\mathrm{Im}h_\#$ is integrable and defines a singular foliation on $M$. Each leaf $L$ carries a pseudo-Riemannian metric $g_L$ given by 
	$g_L(h_\#(\al),h_\#(\be))=h(\al,\be).$ On the other hand, \eqref{eqnabla} shows
	that any leaf $L$   carries an affine structure $\na^L$ and one can check easily that $(\na^L,g_L)$ is a pseudo-Hessian structure on $L$ which is affine special real when $(M,\na,h)$ is. 
\end{proof}

\begin{exem} Consider $\R^n$ endowed with its canonical affine structure $\na$ and denote by $(x_1,\ldots,x_r,y_1,\ldots,y_{n-r})$ its canonical linear coordinates. Let $f\in C^\infty(M)$ such that the matrix $\left(\frac{\partial^2 f}{\partial x_i\partial x_j} \right)$ is invertible  and put
	\[ h=\sum_{i,j=1}^rh_{ij}\partial_{x_i}\otimes\partial_{x_j}, \]where $(h_{ij})$ is the inverse of the matrix $\left(\frac{\partial^2 f}{\partial x_i\partial x_j} \right).$ Then $(\R^n,\na,h)$ is a generalized pseudo-Hessian structure. This is a consequence of the following proposition.

\end{exem}

\begin{pr}\label{prfoliation} Let $(M,\na)$ be an affine manifold and $i:\mathcal{F}\too TM$ a subbundle such that, for any $X,Y\in\Ga(\mathcal{F})$, $\na_XY\in \Ga(\mathcal{F})$. Suppose that there exists $\phi\in C^\infty(M)$ such that $g$ given by $g_x(u,v)=\na_ud\phi(v)$ for any $x\in M$ and any $u,v\in \mathcal{F}_x$ is nondegenerate symmetric bilinear form on $\mathcal{F}_x$. Then $h_\#=i\circ \#\circ i^*$, where $\#:\mathcal{F}^*\too\mathcal{F}$ is the isomorphism associated to $g$, defines a generalized pseudo-Hessian structure on $(M,\na)$. 
\end{pr}

\begin{proof} For any $\al,\be,\ga\in\Om^1(M)$, put $X=h_\#(\al)$, $Y=h_\#(\be)$ and $Z=h_\#(\ga)$. There are three vector fields tangent to $\mathcal{F}$. Since $\na_XY$ and $\na_XZ$ are tangent to $\mathcal{F}$ then
	\[ \na_{h_\#(\al)}h(\be,\ga)=-h_\#(\al).h(\be,\ga)+\prec\be,\na_{h_\#(\al)}h_\#(\ga)\succ +
	\prec\ga,\na_{h_\#(\al)}h_\#(\be)\succ=-\na_Xg(Y,Z).\]
	 Now
	\begin{eqnarray*}
		\na_Xg(Y,Z)-\na_Yg(X,Z)&=&X.\na d\phi(Y,Z)-Y.\na d\phi(X,Z)-\na d\phi(\na_XY-\na_YX,Z)
		-\na d\phi(Y,\na_XZ)+\na d\phi(X,\na_YZ)\\
		&=&[X,Y].d\phi(Z)-X.d\phi(\na_YZ)+Y.d\phi(\na_XZ)-[X,Y].d\phi(Z)+d\phi(\na_{[X,Y]}Z)\\&&
		-Y.d\phi(\na_XZ)+d\phi(\na_Y\na_XZ)
		+X.d\phi(\na_YZ)-d\phi(\na_{X}\na_YZ)\\
		&=&0.\qedhere
	\end{eqnarray*}

\end{proof}

The following proposition is a generalization of  Lemma 2.1 in \cite{furness}. The proof we give here is different. 

\begin{pr}\label{local2} Let $(M,\na)$ be an affine manifold and $i:D\too TM$ a subbundle such that, for any $X,Y\in\Ga(D)$, $\na_XY\in \Ga(D)$. Then, for any $m\in M$, there exists a coordinates system $(x_1,\ldots,x_r,y_1,\ldots,y_{n-r})$ on an open set around $m$ such that, for any $p\in U$,
	\[ D(p)=\mathrm{span}\{\partial_{x_1}(p),\ldots,\partial_{x_r}(p)   \}\esp \quad\na_{\partial_{x_i}}\partial_{x_j}=0,\;i,j=1,\ldots,r. \]
	
\end{pr}

\begin{proof} Denote by $\na^1$   the restriction of $\na$ to $D$. Then $(D,M,i,\na^1)$ is a left symmetric algebroid. Moreover, $\na^1$ satisfies the hypothesis of  Lemma \ref{saïd}. Let $(e_1,\ldots,e_r)$ be a basis of $D(m)$. By virtue of Lemma \ref{saïd}, there exists a family of local vector fields $X_1,\ldots,X_r$ tangent to $D$ such that $\na^1X_i=0$, $X_i(m)=e_i$, $i=1,\ldots,r$. For any, $i,j=1,\ldots,r$, $[X_i,X_j]=\na_{X_i}X_j-\na_{X_j}X_i=0$ and the proposition follows by 	applying Frobenius's Theorem.
	\end{proof}

The following theorem shows that any generalized pseudo-Hessian structure is locally as Proposition \ref{prfoliation} near any regular point. This can be compared to Darboux-Weinstein near a regular point in Poisson geometry.
\begin{theo} \label{local} Let $(M,\na,h)$ be a generalized pseudo-Hessian structure and $x\in M$ such the rank of $h_\#$ is constant in a neighborhood of $x$. Then there exists a chart $(x_1,\ldots,x_r,y_1,\ldots,y_{n-r})$ and a function $f(x,y)$ such that
	\[ h=\sum_{i,j=1}^rh_{ij}\partial_{x_i}\otimes\partial_{x_j}, \quad\na_{\partial_{x_i}}\partial_{x_j}=0,\; i,j=1,\ldots,r, \]and the matrix $(h_{ij})$ is invertible and its inverse is the matrix $\left(\frac{\partial^2 f}{\partial x_i\partial x_j} \right).$
	
\end{theo}
\begin{proof} By applying Proposition \ref{local2} to  $\mathrm{Im}h_\#$ near $x$,  there exists a chart $(x_1,\ldots,x_r,y_1,\ldots,y_{n-r})$ such that
	\[ \mathrm{span}\left(\partial_{x_1},\ldots,\partial_{x_r} \right)=\mathrm{Im}h_\#\esp \na_{\partial_{x_i}}\partial_{x_j}=0,\; i,j=1,\ldots,r. \] This implies that $h_\#(dy_i)=0$ for $i=1,\ldots,n-r$. Note first that, for any $i,j,k=1,\ldots,n$, $\prec \na_{\partial_{x_k}}^*dx_j,\partial_{x_j}\succ=0$ and hence
	from \eqref{connectionh}, 
	\begin{eqnarray*}
	\prec \D_{dx_i}dx_j,\partial_{x_k}\succ&=&\na_{\partial_{x_k}}h(dx_i,dx_j)+\prec \na_{h_\#(dx_i)}dx_j,\partial_{x_k}\succ\\
	&=&\prec dh(dx_i,dx_j),\partial_{x_k}\succ-\prec \na_{\partial_{x_k}}^*dx_i,h_\#(dx_j)\succ-\prec \na_{\partial_{x_k}}^*dx_j,h_\#(dx_k)\succ\\
	&=&\prec dh(dx_i,dx_j),\partial_{x_k}\succ.
	\end{eqnarray*}Thus $\D_{dx_i}dx_j=dh(dx_i,dx_j)+\al_{ij}$ where $\al_{ij}\in\ker h_\#$.
	This implies that $[h_\#(dx_i),h_\#(dx_j)]=h_\#(\D_{dx_i}dx_j)-h_\#(\D_{dx_j}dx_i)=0$.
	 So there exists a coordinates system $(z_1,\ldots,z_n)$ such that
	\[  h_\#(dx_i)=\partial_{z_i},\; i=1,\ldots,r.\] We deduce that
	\[ \partial_{x_i}=\sum_{j=1}^rh^{ij}\partial_{z_j},\; i=1,\ldots,r,\; (h^{ij})=(dh(dx_i,dx_j))^{-1}. \]
	Thus
	$ h^{ij}=\frac{\partial z_j}{\partial x_i}.$
	We consider $\sigma=\sum_{j=1}^rz_jdx_j$. We have $d_{\mathcal{F}}\sigma=0$ so,
	according to  the foliated Poincar\'e Lemma (see \cite{calvin} pp. 56), there exists a function $f$ such that
	$h^{ij}=\frac{\partial^2f}{\partial x_i\partial x_j}$, which completes the proof.
\end{proof}

\begin{remark} It is important to generalize the theorem above near a singular point. The authors have no ideas how to do it and left the problem open.
	
\end{remark}

\section{Linear generalized pseudo-Hessian structures}\label{section7}
In this section, we will pursue the study of similarities between pseudo-Hessian manifold and Poisson manifolds. Namely, we will show that the dual of any commutative and associative algebra carries a canonical pseudo-Hessian structure and we will study these structures in detail.

Let $(M,\na,h)$ be a generalized pseudo-Hessian manifold. Let $x\in M$ and denote by $\G_x=\ker h_{\#}(x)$.
Let $\D$ be the right-anchored product associated to $h$  given by \eqref{connectionh}. According to \eqref{eqnabla}, for any $\al,\be\in\Om^1(M)$, $h_\#(\D_\al\be)=\na_{h_\#(\al)}h_\#(\be)$. This shows that if $h_\#(\al)(x)=0$ then $h_\#(\D_\al\be)(x)=0$. Moreover, $\D_\al\be-\D_{\be}\al=\na_{h_\#(\al)}\be-\na_{h_\#(\be)}\al$. This implies that if $h_\#(\al)(x)=h_\#(\be)(x)=0$ then $\D_\al\be(x)=\D_\be\al(x)$.
For any $a,b\in\G_x$ put
\[ a.b=(\D_\al\be)(x), \]where  $\al,\be$ are 2 differential 1-forms satisfying $\al(x)=a$ and $\be(x)=b$. According to the what above this defines a commutative product on $\G_x$ and moreover, by using the vanishing of the curvature of $\D$, we get:
\begin{pr} $(\G_x,.)$ is a commutative associative algebra.

\end{pr}

As the dual of a Lie algebra carries a natural Poisson structure, the dual of a commutative associative algebra carries a generalized pseudo-Hessian structure.  Indeed, let $(\mathcal{A},.)$ be a finite dimensional commutative associative algebra. We define a symmetric bivector $h$ on $\mathcal{A}^*$ by putting
\begin{equation}\label{h} h(\al,\be)(\mu)=\prec\mu,\al(\mu).\be(\mu)\succ,\quad \al,\be\in\Om^1(\mathcal{A}^*)=C^\infty(\mathcal{A}^*,\mathcal{A}),\mu\in \mathcal{A}^*. \end{equation}We denote by $\na^0$ the canonical affine connection  of $\mathcal{A}^*$ given by $\na^0_{X}Y(\mu)=d_\mu Y(X(\mu))$ where $X,Y:\mathcal{A}^*\too \mathcal{A}^*$ are regarded as vector fields on $\mathcal{A}^*$. For any $u\in \mathcal{A}$, we denote by $u^*$ the linear function on $\mathcal{A}^*$ given by $u^*(\mu)=\prec\mu,u\succ$,  by $X_u$ the vector field on $\mathcal{A}^*$ given by $X_u=h_\#(du^*)$ and by $L_u:\mathcal{A}\too\mathcal{A}$ the left multiplication by $u$. Let $\D$ be the right-anchored product associated to $(\mathcal{A}^*,\na^0,h)$ and given by \eqref{connectionh}. Finally, denote by $T$ the tensor field on $\mathcal{A}^*$ given by $T(\al,\be,\ga)=\na^0_{h_\#(\al)}h(\be,\ga)$. A straightforward computation gives the following proposition.
\begin{pr}\label{formulas} For any $u,v,w,x\in \mathcal{A}$, we have
	\begin{eqnarray*} h(du^*,dv^*)&=&(u.v)^*,\;X_u=L_u^*,\; \na^0_{X_u}X_v=X_{u.v},\;T(du^*,dv^*,dw^*)=(u.v.w)^*,\;\\
	\D_{du^*}dv^*&=&d(u.v)^*\esp \D T(du^*,dv^*,dw^*,dx^*)=-2(u.v.w.x)^*. \end{eqnarray*}
	
\end{pr}
As a consequence of these formulas and Theorem \ref{poissonbis}, we get the following result.
\begin{theo}\label{main} $(\mathcal{A}^*,\na^0,h)$ is a generalized pseudo-Hessian manifold and the singular foliation associated to $\mathrm{Im}h_\#$ is given by the orbits of the linear action $\Phi$ of the abelian Lie group $(\mathcal{A},+)$ on $\mathcal{A}^*$ given by $\Phi(u,\mu)=\exp(L_u^*)(\mu)$. Moreover,   $(\mathcal{A}^*,\na^0,h)$ is a generalized affine special real manifold if and only if $\mathcal{A}^4=0$. In particular, the orbits of $\Phi$ are pseudo-Hessian manifolds and if $\mathcal{A}^4=0$ they are affine special real manifolds.
	
\end{theo}

\begin{remark} This theorem is similar to the well-known result asserting that the dual of a Lie algebra carries a natural Poisson structure. However, there is an important difference between the two situations. In the case of a Lie algebra, the symplectic leaves are the orbits of the co-adjoint action of any connected Lie group associated to the Lie algebra and the action preserves the symplectic form of any leaf. In the case of a commutative associative algebra, the pseudo-Hessian leaves are the orbits of the action of $(\mathcal{A},+)$ and this action preserves the affine structure of any leaf but not its pseudo-Hessian metric unless $\mathcal{A}^3=0$. Not that these pseudo-Hessian manifolds are diffeomorphic to a $\R^q\times \mathbb{T}^p$.
	
\end{remark}

The class of associative commutative algebras constitutes a large class of non associative algebras so Theorem \ref{main} is a powerful tool to build examples of pseudo-Hessian manifolds and affine special real manifolds. Since any pseudo-Hessian structure on a manifold gives rise to a pseudo-K\"ahlerian structure on its tangent bundle we get also a machinery to build examples of
pseudo-K\"ahlerian manifolds.
In what follows, we will illustrate this by showing that, using Theorem \ref{main}, we can get interesting examples. Namely, we will show that the Hessian curvature of these manifolds is not trivial in general. Shima introduced the notion of Hessian curvature, which is a finer invariant than Riemannian curvature and is related with the curvature of the associated K\"ahler metric on  the total space of the tangent bundle. Let us  recall first the definition of the Hessian curvature and  the definitions of some basic notions in a pseudo-Hessian manifold (see \cite{shima} for more details). 

Let $(M,\na,g)$ be a pseudo-Hessian manifold. 
Denote by $D$ the Levi-Civita connection of $g$ and put $\na'=2D-\na$ and $\ga=D-\na$. The connection $\na'$ is called the dual connection of $\na$ with respect to $g$ and $(M,\na',g)$ is also a pseudo-Hessian structure. 
The Hessian curvature $(M,\na,g)$ is the tensor $Q$ given by
$ Q=\na\ga$.
The first and the second Koszul forms are given, respectively, by
$\al(X)=\tr(i_X\ga)$ and $\be=\na\al.$ 

Let compute now all the mathematical objects above in the case where $M=\{\exp(L_a^*)(\mu),a\in \mathcal{A}  \}$ is an orbit of the pseudo-Hessian foliation associated to the pseudo-Hessian manifold $(\mathcal{A}^*,\na^0,h)$ appearing in Theorem \ref{main}. Note that $T_\nu M=\{ X_a(\nu),a\in \mathcal{A}  \}$. As above, we denote by $\na$ the affine connection on $M$, $g$ the pseudo-Riemannian metric, $D$ the Levi-Civita connection and so on.
The following proposition is a consequence of an easy and straightforward computation.
\begin{pr}\label{prc}
For any $a,b,c\in \mathcal{A}$ and any $\nu\in \mathcal{A}^*$,
\begin{eqnarray*} g(X_a(\nu),X_b(\nu))&=&\prec \nu,a.b\succ,\; \na_{X_a}X_b=X_{a.b}, D_{X_a}X_b=\frac12X_{a.b}, \; \na'_{X_a}X_b=0,\; \\\;Q(X_a,X_b)X_c&=&\frac12X_{a.b.c},\;\al(X_a)=-\frac12 \tr( L_a)\esp\be(X_a,X_b)=\frac12\tr(L_{a.b}). \end{eqnarray*}
 In particular, $g$ is a flat pseudo-Riemannian metric and $Q=0$ if and only if  $\mathcal{A}^4=0$.

\end{pr}

We end this paper by considering examples of commutative associative algebras.  For each of them we choose an orbit $M$ and give in  an affine  system of coordinates $(x_i)$  the pseudo-Hessian metric $g$ and a function $\phi$ such that $g_{ij}=\frac{\partial^2 \phi}{\partial x_i\partial x_j}$. Some examples comes from the lists of low dimensional associative algebras obtained in \cite{rhakimov}.

\begin{exem}\begin{enumerate} All the algebras bellow are identified with $\R^n$ with its canonical basis $(e_i)_{i=1}^n$ and $(e_i^*)_{i=1}^n$ is the dual basis. The action $\Phi$ of $\mathcal{A}$ on $\mathcal{A}^*$ is given by
		$\Phi(a,\mu)=\exp(L_{a}^*)(\mu)$ and, for any $a\in \mathcal{A}$, $X_a$ is the vector fields on $\mathcal{A}^*$ given by $X_a=L_a^*$, where $L_a$ is the left multiplication by $a$. We denote by $\na$ the canonical connection on $\mathcal{A}^*$.

		\item We take $\mathcal{A}=\R^n$ as a product of $n$ copies of the associative commutative algebra $\R$. The product is given by $e_ie_i=e_i$ for $i=1,\ldots,n$. We denote by $(a_i)_{i=1}^n$ the linear coordinates of $\mathcal{A}$ and $(x_i)_{i=1}^n$ the dual coordinates on $\mathcal{A}^*$. We have
		\[ \Phi\left(\sum_{i=1}^na_ie_i,\sum_{i=1}^nx_ie_i^*\right)=\sum_{i=1}^n e^{a_i}x_ie_i^*. \]Moreover, for any $i=1,\ldots,n$, $X_{e_i}=x_i\partial_{x_i}$. The orbit of a point $x\in \mathcal{A}^*$ is $M_x=\left\{ \sum_{i=1}^n e^{a_i}x_ie_i^*,a_i\in\R \right\}$. It is a convex cone    and one can see easily that if  $\phi:\mathcal{A}^*\too \R$ is the function given by
		\[ \phi(u)=\sum_{i=1}^nu_i\ln|u_i|, \]then the restriction of $\na d\phi$ to $M_x$ together with the restriction of $\na$ to $M_x$ define the pseudo-Hessian structure on $M_x$ described in Theorem \ref{main}. Note here that the signature of the pseudo-Hessian metric on $M_x$ is exactly $(p,q)$ where $p$ is the number of $x_i$ such that $x_i>0$ and $q$ is the number of $x_i$ such that $x_i<0$. Note that if $x_i>0$ for $i=1,\ldots,n$ then the metric on $M_x$ is definite positive and we recover the example given in \cite{shima} pp. 17.
		
		\item We take $\mathcal{A}=\C$ endowed with its canonical structure of commutative and associative algebra.  We have
	\[ e_1.e_1=e_1,\; e_1.e_2=e_2.e_1=e_2,\; e_2.e_2=-e_1. \]We denote here by $(x,y)$ the linear coordinates on $\mathcal{A}$ associated to $(e_1,e_2)$ and $(\al,\be)$ the dual coordinates on $\mathcal{A}^*$. We have
	\[ X_{e_1}=\al\partial_{\al}+\be\partial_{\be}\esp X_{e_2}= \be\partial_{\al}-\al\partial_{\be} \] and it is easy to check that
	\[ \Phi(xe_1+ye_2,\al e_1^*+\be e_2^*)=e^x\left( (\al\cos(y)+\be\sin(y))e_1^*+(-\al\sin(y)+\be\cos(y))e_2^*\right). \]
	We deduce that we have two orbits the origin and $\mathcal{A}^*\setminus\{0\}$. Let describe the pseudo-Hessian structure of $M:=\mathcal{A}^*\setminus\{0\}$. The pseudo-Hessian metric $g$ satisfies
		\[ g(X_{e_1},X_{e_1})=\al,\;g(X_{e_1},X_{e_2})=\be,\; g(X_{e_2},X_{e_2})=-\al \]and hence 
	\[ g=\frac{1}{\al^2+\be^2}(\al d\al^2+2\be d\al d\be-\al d\be^2). \]
	Thus $(M,\na ,g)$ is a Lorentzian Hessian manifold. Moreover, the metric $g$ is flat.
	Now we look for a function $f$ on $M$ such that $g=\na df$, i.e.,
	\[ \frac{\partial^2 f}{\partial\al^2}=\frac{\al}{\al^2+\be^2},\;\frac{\partial^2 f}{\partial\be^2}=\frac{-\al}{\al^2+\be^2}\esp \frac{\partial^2 f}{\partial\al\partial\be}=\frac{\be}{\al^2+\be^2}  \] The function $f$ given by
	\[ f(\al,\be)=\frac12\al\ln(\al^2+\be^2)+\be\arctan\left(\frac{\al}{\be} \right) \]satisfies these equations on the open set $\{ \be\not=0\}$. Note that this function is harmonic.

	\item We take $\mathcal{A}=\R^3$ with the commutative associative product given by $e_1e_1=e_2$ and $e_1e_2=e_3$. We have $\mathcal{A}^3\not=0$ and $\mathcal{A}^4=0$. We denote by $(a,b,c)$ the linear coordinates of $\mathcal{A}$ and $(x,y,z)$ the dual coordinates of $\mathcal{A}^*$. We have
	\[ X_{e_1}=y\partial_{x}+z\partial_y,\; X_{e_2}=z\partial_x\esp X_{e_3}=0 \]and
	\[ \Phi(ae_1+be_2+ce_3,xe_1^*+ye_2^*+ze_3^*)=(x+ay+(\frac12a^2+b)z,y+az,z). \]
	The orbits of this action are the plans $\{z=c, c\not=0\}$, the lines $\{ z=0,y=c,c\not=0\}$ and the points $\{ (c,0,0)\}$. The pseudo-Riemannian metric on $M_c=\{z=c, c\not=0\}$ is given by $$g_c(X_{e_1},X_{e_1})=y,\;g_c(X_{e_1},X_{e_2})=c\esp g_c(X_{e_2},X_{e_2})=0.$$This is a Lorentzian metric and one can check easily that, if $\phi(x,y,z)=-\frac{y^3}{6z^2}+\frac{xy}{z}$ then $g_c$ is the restriction of $\na d\phi$ to $M_c$. Note that since $\mathcal{A}^4=0$ then $(M_c,\na_{|M_c},g_c)$ is an affine special real manifold.
	However, the pseudo-Hessian metric on the line $L_c=\{ z=0,y=c,c\not=0\}$ is given by the restriction of $\na d\phi_1$ where $\phi_1(x,y,z)=\frac{x^2}{2y}$.
\item	We take $\mathcal{A}=\R^3$ with the commutative associative product given by
	\[ e_1e_1=e_2,\; e_1e_3=e_1,\; e_2e_3=e_2,\; e_3e_3=e_3. \]
	We denote by $(a,b,c)$ the linear coordinates on $\mathcal{A}$ and $(x,y,z)$ the dual coordinates on $\mathcal{A}^*$. We have
	\[ \Phi(ae_1+be_2+ce_3,xe_1^*+ye_2^*+ze_3^*)=e^c\left(x+ay,y, ax+\frac12(a^2+2b)y+z  \right). \]The orbits have dimension 3,2,1 or 0. The three dimensional orbits are $\{y>0\}$ and $\{y<0\}$. The two dimensional orbits are $\{y=0,x>0\}$ and $\{y=0,x<0\}$. The one dimensional orbits are $\{y=x=0,z>0\}$ and $\{y=x=0,z<0\}$. The origin is the only zero dimensional orbit. Let describe the pseudo-Hessian structure on $M=\{y>0\}$ or  $M=\{y<0\}$. We have
	\[ X_{e_1}=y\partial_x+x\partial_z,\; X_{e_2}=y\partial_z\esp X_{e_3}=x\partial_x+y\partial_y+z\partial_z, \]
	and the pseudo-Hessian metric $g$ on $M$ is satisfies
	\[ g(X_{e_1},X_{e_1})=y,\;g(X_{e_1},X_{e_2})=0,\; g(X_{e_1},X_{e_3})=x,\;
	g(X_{e_2},X_{e_2})=0,\;g(X_{e_2},X_{e_3})=y\esp g(X_{e_3},X_{e_3})=z. \]
	Note that the matrix of $g$ in $(X_{e_1},X_{e_2},X_{e_3})$ is just the passage matrix $P$ from
	$(X_{e_1},X_{e_2},X_{e_3})$ to $(\partial_x,\partial_y,\partial_z)$ and hence the matrix of $g$ in $(\partial_x,\partial_y,\partial_z)$ is $P^{-1}$. Thus,
	in the coordinates $(x,y,z)$, we have
	\[ g=\frac1y\left( dx^2+\frac{x^2-yz}{y}dy_2+2dydz-\frac{2x}{y}dxdy\right).    \]
	One can check easily that $g$ is the restriction of $\na d\phi$ where $\phi(x,y,z)=z\ln|y|+\frac{x^2}{2y}$. This metric is of signature $(+,+,-)$ in $\{y>0\}$ and $(+,-,-)$ in $\{y<0\}$.
	
	\item	We take $\mathcal{A}=\R^4$ with the commutative associative product given by
	\[ e_1e_1=e_2,\; e_1e_2=e_3,\; e_1e_3= e_2e_2=e_4. \]We have $\mathcal{A}^3\not=0$ and $\mathcal{A}^4=0$.
	We denote by $(a,b,c,d)$ the linear coordinates on $\mathcal{A}$ and $(x,y,z,t)$ the dual coordinates on $\mathcal{A}^*$. We have
	\[ \Phi(ae_1+be_2+ce_3+de_4^*,xe_1^*+ye_2^*+ze_3^*+te_4^*)=(x+ay+(\frac12a^2+b)z+(\frac16a^3+ab+c)t,y+az+(\frac12a^2+b)t,z+at,t) \]and
	\[ X_{e_1}=y\partial_x+z\partial_y+t\partial_z,\; X_{e_2}=z\partial_x+t\partial_y,\; X_{e_3}=t\partial_x\esp X_{e_4}=0. \]
	Let describe the pseudo-Hessian structure of the hyperplan $M_c=\{t=c,c\not=0\}$ endowed with the coordinates $(x,y,z)$. Since the matrix of $g_c$ in $(X_{e_1},X_{e_2},X_{e_3})$ is the passage matrix $P$ from  $(X_{e_1},X_{e_2},X_{e_3})$ to $(\partial_x,\partial_y,\partial_z)$, we get
	\[ g_c=\frac1{c}\left(2dxdz+dy^2-\frac{2z}{c}dydz+\frac{(z^2-yc)}{c^2}dz^2  \right). \]
	The signature of this metric is $(+,+,-)$ if $c>0$ and $(+,-,-)$ if $c<0$. One can check easily that $g_c$ is the restriction of $\na d\phi$ to $M_c$, where
	\[ \phi(x,y,z,t)=\frac{z^4}{12t^3}+\frac{y^2}{2t}-\frac{z^2y}{2t}+\frac{xz}{t}. \]Since $\mathcal{A}^4=0$, $M_c$ is an affine special real manifold.
	\item	We take $\mathcal{A}=\R^4$ with the commutative associative product given by
	\[ e_1e_1=e_1,\; e_1e_2=e_2,\;e_1e_3=e_3,\;e_1e_4=e_4,\;e_2e_2=e_3,\; e_2e_3=e_4. \]We have
	\[ X_{e_1}=x\partial_x+y\partial_y+z\partial_z+t\partial_t,\;X_{e_2}=y\partial_x+z\partial_y+t\partial_z,\; X_{e_3}=z\partial_x+t\partial_y\esp X_{e_4}=t\partial_x. \]
	Thus $\{t>0\}$ and $\{t<0\}$ are orbits and hence carry a pseudo-Hessian structures. Let us determine the pseudo-Hessian metric. The same argument as above gives that the metric is given by the inverse of the passage matrix from $(X_{e_1},\ldots,X_{e_4})$ to $(\partial_x,\partial_y,\partial_z,\partial_t)$. Thus
	\[ g=\frac1t\left( 2dxdt+2dydz-\frac{2z}{t}dydt-\frac{z}{t}dz^2+\frac{2(z^2-yt)}{t^2}dzdt+\frac{2zyt-xt^2-z^3}{t^3}dt^2  \right). \]The signature of this metric is $(+,+,-,-)$. One can check easily that $g$ is the restriction of $\na d\phi$ to $M$, where
	\[ \phi(x,y,z,t)=-\frac{z^3}{6t^2}+\frac{yz}{t}+x\ln|t|. \]

\end{enumerate}	
\end{exem}

\bibliographystyle{elsarticle-num}

\end{document}